\documentclass[reqno]{amsart}
\usepackage{hyperref}
\usepackage{cleveref}

\usepackage{mathtools}
\usepackage{listings}
\usepackage{minted}
\usepackage{amsmath}
\usepackage{amssymb}
\usepackage{float}
\usepackage{tabularx}
\usepackage{bm}
\def\bb{\mathbb}

\def\bbN{\bb{N}}
\def\bbZ{\bb{Z}}

\DeclareMathSymbol{<}{\mathrel}{letters}{'074}
\DeclareMathSymbol{>}{\mathrel}{letters}{'076}

\usepackage{graphicx}
\DeclareUnicodeCharacter{2202}{\texttt{\reflectbox{6}}}

\usepackage{amsthm}
\usepackage{thmtools}

\DeclareMathOperator{\Ann}{Ann}

\DeclareMathOperator{\Soc}{Soc}

\DeclareMathOperator{\row}{Row}

\usepackage{xparse}

\NewDocumentCommand{\Mod}{O{}O{}}{\prescript{}{#1}{\mathbf{Mod}}_{#2}}

\newtheorem{thm}{Theorem}[section]
\newtheorem{cor}[thm]{Corollary}
\newtheorem{lem}[thm]{Lemma}
\newtheorem{prop}[thm]{Proposition}
\newtheorem{defn}[thm]{Definition}

\newtheorem{ques}[thm]{Question}

\theoremstyle{remark}
\newtheorem{rem}[thm]{Remark}

\newtheorem{exa}[thm]{Example}

\usepackage[shortlabels]{enumitem}
\setlist{noitemsep}

\usepackage{xfrac}
\usepackage{quiver}
\usepackage{tcolorbox}

\definecolor{myred}{HTML}{E74C3C}
\tcbset{width=\textwidth,boxrule=0pt,colback=myred,arc=0pt,auto outer arc,left=0pt,right=0pt,top=0pt,bottom=0pt,boxsep=5pt}

\usepackage{ytableau}
\usetikzlibrary{calc}
\tikzset{caption/.style={execute at end picture={\path
let \p1=($(current bounding box.east)-(current bounding box.west)$) in
(current bounding box.south) node[below,text width=\x1-4pt,align=center] 
{\vphantom{$|$}#1};}}}

\newcounter{x}
\newcounter{y}
\newcounter{z}
\newcommand\xaxis{210}
\newcommand\yaxis{-30}
\newcommand\zaxis{90}
\newcommand\topside[4]{
  \fill[fill=#4,fill opacity=1, draw=black,shift={(\xaxis:#1/2)},shift={(\yaxis:#2/2)},
  shift={(\zaxis:#3/2)}] (0,0) -- (30:1/2) -- (0,1/2) --(150:1/2)--(0,0);
}
\newcommand\leftside[4]{
  \fill[fill=#4!90!black,fill opacity=1, draw=black,shift={(\xaxis:#1/2)},shift={(\yaxis:#2/2)},
  shift={(\zaxis:#3/2)}] (0,0) -- (0,-1/2) -- (210:1/2) --(150:1/2)--(0,0);
}
\newcommand\rightside[4]{
  \fill[fill=#4!80!black,fill opacity=1, draw=black,shift={(\xaxis:#1/2)},shift={(\yaxis:#2/2)},
  shift={(\zaxis:#3/2)}] (0,0) -- (30:1/2) -- (-30:1/2) --(0,-1/2)--(0,0);
}
\newcommand\cube[4]{
  \topside{#1}{#2}{#3}{#4} \leftside{#1}{#2}{#3}{#4} \rightside{#1}{#2}{#3}{#4}
}
\newcommand\planepartition[3][0]{
 \setcounter{x}{-1}
  \foreach \a in {#2} {
        \addtocounter{x}{1}
        \setcounter{y}{-1}
            \foreach \b in \a {
            \addtocounter{y}{1}
            \setcounter{z}{-1}
            \addtocounter{z}{#1} 
            \ifnum \b>0
            \foreach \c in {1,...,\b} {
                \addtocounter{z}{1}
                \cube{\value{x}}{\value{y}}{\value{z}}{#3}
      }\fi
    }
  }
}

\newcommand\axes[3]{
  \draw[draw=black,line cap=round] (0,0) -- (150:-#1/2);
  \node at (150:-#1/2-0.25) {$x_1$};
  \draw[draw=black,line cap=round] (0,0) -- (30:-#2/2);
  \node at (30:-#2/2-0.25) {$x_2$};
  \draw[draw=black,line cap=round] (0,0) -- (0,#3/2);
  \node at (0,#3/2+0.25) {$x_3$};
}

\newcommand\numberedbox[4]{
    \draw[fill=#4, draw=black] (#1 * 0.5,#2 * 0.5) -- (#1 * 0.5 + 0.5,#2 * 0.5) -- (#1 * 0.5+0.5,#2 * 0.5+0.5) -- (#1 * 0.5,#2 * 0.5 +0.5) -- (#1 * 0.5,#2 * 0.5);
    \draw (#1 * 0.5 + 0.25,#2 * 0.5 + 0.25) node {#3};
}

\newcommand\emptybox[3]{
    \filldraw[fill=#3, draw=black] (#1*0.5,#2*0.5) rectangle (#1*0.5+0.5,#2*0.5+0.5);
}

\newcommand\tableau[4]{
\setcounter{x}{#2-1}
    \foreach \a in {#1} {
    \addtocounter{x}{1}
    \setcounter{y}{#3-1}
        \foreach \b in {1,...,\a} {
            \addtocounter{y}{1}
            \emptybox{\value{x}}{\value{y}}{#4}
        }
    }
}

\newcommand\axestwo[2]{
  \draw[draw=black,line cap=round] (0,0) -- (#1/2,0);
  \node at (#1/2+0.25,0) {$x_1$};
  \draw[draw=black,line cap=round] (0,0) -- (0,#2/2);
  \node at (0,#2/2+0.25) {$x_2$};
}

\title{On combinatorial algebras generated by three commuting matrices}
\author{Ron Cherny, Tam An Le Quang, Matthew Satriano}

\begin{document}

\begin{abstract}
Motzkin and Taussky (and independently, Gerstenhaber) proved that the unital algebra generated by a pair of commuting $d\times d$ matrices over a field has dimension at most $d$. Since then, it has remained an open problem to determine whether the analogous statement is true for triples of matrices which pairwise commute. We answer this question for combinatorially-motivated classes of such triples.
\end{abstract}

\maketitle
\tableofcontents

\section{Introduction}

Let $k$ be a field and consider the following question.

\begin{ques}\label{GQ}
    Let $A_1, \cdots, A_n \in M_d(k)$ be pairwise commuting matrices. Is $\dim_k(k[A_1, \cdots, A_n]) \leq d$?
\end{ques}

For $n = 1$, the statement is true by the Cayley-Hamilton theorem since every matrix satisfies its own characteristic polynomial. For $n = 2$, the statement is true as shown by Motzkin--Taussky \cite{MTT} as well as Gerstenhaber in \cite{Ger61}. For $n \geq 4$, the statement is false; an easy counterexample when $n=4$ is given by $\dim_k(k[E_{13}, E_{14}, E_{23}, E_{24}]) = 5$, where $E_{ij}$ form the standard basis for $M_4(k)$.

The $n = 3$ case of Question \ref{GQ} has been a longstanding open question. Both geometric and algebraic techniques have been used to address it: to begin, the $n=2$ case was handled by considering the algebraic variety $C(2,d)$ parameterizing pairs $(A_1,A_2)$ of commuting $d\times d$ matrices. It was shown in \cite{MTT} that $C(2,d)$ is irreducible, allowing one to reduce to the case of generic pairs $(A_1,A_2)$, which are then diagonalizable and thus easy to handle. When $n\geq 3$, however, the variety $C(n,d)$ of $n$ pairwise commuting matrices is notoriously complicated. For $n\geq4$ and $d\geq4$, $C(n,d)$ has multiple irreducible components \cite{Ger61,Guralnick92}. $C(3,d)$ is irreducible for $d\leq 10$ \cite{sivicII}, and $C(3,d)$ is reducible for $d\geq 29$ \cite{HolOmla,NgoSivic14}. See also \cite{JelisiejewSivic22} for further results on the structure of components of $C(n,d)$. This makes it essentially intractable to address Question \ref{GQ} through a study of $C(3,d)$.

Other approaches to Question \ref{GQ} have been commutative algebraic \cite{Wadsworth, Bergman} and linear algebraic \cite{BH, LL}. Question \ref{GQ} is also known when one imposes certain linear algebraic constraints such as assuming $A_1$ has nullity at most $3$ \cite{GS,sivicII}.  We refer to \cite{SethurSurvey, HolbrookOmeara} for a survey of further results.

\vspace{0.5em}

In this paper, the approach we take is combinatorial. To motivate this, we first note that a straightforward reformulation of Question \ref{GQ} can be given in terms of $S$-modules $N$ with $\dim_k(N)=d$, where $S=k[x_1,\dots,x_n]$. The connection with matrices is given by fixing a $k$-basis of $N$ and letting $A_i$ be the matrix obtained through multiplication by $x_i$ on $N$:

\begin{prop}[see, e.g., {\cite[Proposition 1.3]{RS18}}]\label{GQM}
    \Cref{GQ} is true if and only if for all $S$-modules $N$ that are finite-dimensional over $k$, $\dim_k(S / \Ann(N)) \leq \dim_k(N)$.
\end{prop}

From the module perspective, one may break up Question \ref{GQ} according to the number of generators of $N$. When $N$ has a single generator, it is of the form $S/I$, where this case is obvious as $S/\Ann(N)=S/I=N$. Thus, the first non-trivial case is when $N$ has two generators. Previous work in \cite{RSS20} and \cite{CSS24} considered special combinatorially-motivated modules of this form, which we now describe:





\begin{defn}[{\cite[\S4.2]{RSS20}}, {\cite[Definition 2]{CSS24}}]\label{def:comb-mods}
    Let $I \subseteq K \subseteq S$ and $J \subseteq L \subseteq S$ be monomial ideals of finite colength. Suppose that $\phi : K / I \to L / J$ is a monomial map and an isomorphism of $S$-modules. Then, call the $S$-module $$\frac{S / I \oplus S / J}{\langle(f, -\phi(f)) : f \in K / I\rangle}$$ the module arising from $(I, J, K, L, \phi)$, and say that it is obtained by gluing $S / I$ and $S / J$ along the $S$-module $K / I \cong L / J$. For $N$ any $S$-module, say $N$ is a combinatorial module if $N$ is isomorphic to the module arising from $(I, J, K, L, \phi)$.
\end{defn}

\begin{rem}
    We wish to emphasize that despite the simplicity of these modules, answering Question \ref{GQ} in this context is far from obvious. Indeed, the class of modules in Definition \ref{def:comb-mods} is robust enough to yield a counterexample to Question \ref{GQ} when $n=4$:~one can take
    \begin{align*}
        I&=(x_1,x_2)^2+(x_3,x_4), \\
        J&=(x_1,x_2)+(x_3,x_4)^2, \text{ and} \\
    M&=(S/I\oplus S/J)/\langle(x_1,0)-(0,x_3), (x_2,0)-(0,x_4)\rangle.
    \end{align*}
\end{rem}

We say a combinatorial module is a \emph{counterexample} if it fails the inequality in \Cref{GQM}. Our main result is:


\begin{thm}\label{BIGMAIN}
    Let $N$ be the combinatorial module arising from $(I, J, K, L, \phi)$. Write the $\bbN^3$-graded indecomposable decomposition of the $S$-module $K / I$
    as $$K / I \cong P_1 \oplus \cdots \oplus P_m.$$ Suppose that for all $j$, the embedding of $P_j$ in $S / I$ is not contained in the ideal $(x_3)$. Then, $N$ is not a counterexample.
\end{thm}

Note that there is no restriction on the form $K/I$ must take. This is in contrast to the previously known results in \cite[Theorem 4]{RSS20} and \cite[Theorem 7]{CSS24}: in the former, $K/I$ was required to be of the form $(S / (x_1, x_2, x_3))^{\oplus m}$, which was generalized in the latter reference to include $K/I$ of the form $$\bigoplus_{j = 1}^m S / (x_1, x_2, x_3^{n_j}).$$

\noindent\textbf{Acknowledgments.} We are indebted to Jenna Rajchgot and Oliver Pechenik for their many insights into this problem. We thank Matt Szczesny for helpful discussions.

\section{Combinatorial setup}

We begin by reframing the question combinatorially, following closely the presentation of \cite[\S4]{RSS20} and \cite[\S2]{CSS24}. Let $\leq$ be the standard lattice on $\bbN^n$.

\begin{defn}
    Recall that a Young diagram is a finite subset of $\bbN^n$ closed under $\leq$ and that a skew Young diagram is the difference of two Young diagrams. We will refer to Young diagrams and skew Young diagrams as standard shapes and skew shapes respectively.
\end{defn}

\begin{rem}
    Note that every standard shape is a skew shape. Every skew shape $\sigma$ can be written as $\overline{\sigma} \smallsetminus \underline{\sigma}$, where with $\overline{\sigma}$ the closure of $\sigma$ under $\leq$ and $\underline{\sigma}$ the complement of $\sigma$ in $\overline{\sigma}$, both $\overline{\sigma}$ and $\underline{\sigma}$ are standard shapes.
\end{rem}

\begin{exa}
    With $n = 2$, we depict an example skew shape $\sigma$ and its resulting $\overline{\sigma}$ and $\underline{\sigma}$.
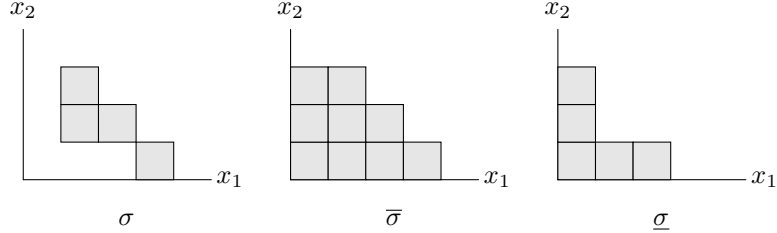
    \begin{center}
        \begin{tikzpicture}[caption=$\sigma$]
            \axestwo{5}{4}
            \emptybox{3}{0}{gray!20!white}
            \emptybox{1}{1}{gray!20!white}
            \emptybox{2}{1}{gray!20!white}
            \emptybox{1}{2}{gray!20!white}
        \end{tikzpicture}
        \begin{tikzpicture}[caption=$\overline{\sigma}$]
            \axestwo{5}{4}
            \tableau{3,3,2,1}{0}{0}{gray!20!white}
        \end{tikzpicture}
        \begin{tikzpicture}[caption=$\underline{\sigma}$]
            \axestwo{5}{4}
            \tableau{3,1,1}{0}{0}{gray!20!white}
        \end{tikzpicture}
    \end{center}
\end{exa}

\begin{defn}
    Let $e_1, \cdots, e_n$ be the standard basis vectors of $\bbN^n$. Given a skew shape $\sigma$, let $\sim$ be an equivalence relation on $\sigma$, where for $v, w \in \sigma$, $v \sim w$ if and only if there is a sequence $v = v_0, \cdots, v_m = w \in \sigma$ where for $j \in [m]$, $v_j - v_{j - 1} \in \{\pm e_1, \cdots, \pm e_n\}$. Call the sequence a path of length $m$ from $v$ to $w$ in $\sigma$ and call the equivalence classes the connected components of $\sigma$. Say $\sigma$ is connected if it has a single connected component.
\end{defn}

\begin{defn}
    Say two skew shapes $\sigma$ and $\tau$ are translationally equivalent if there is some $v \in \bbZ^n$ with $\sigma + v = \tau$.
\end{defn}

\begin{rem}
    Given a nonempty skew shape $\sigma$, denote by $\wedge\sigma$ the meet of $\sigma$ in the lattice $\bbN^n$. Observe that $\sigma - \wedge\sigma$ is a skew shape translationally equivalent to $\sigma$, and that two nonempty skew shapes $\sigma$ and $\tau$ are translationally equivalent if and only if $\sigma - \wedge\sigma = \tau - \wedge\tau$.
\end{rem}

\begin{exa}
    With $n = 2$, we depict an example skew shape $\sigma$ and its resulting $\sigma - \wedge\sigma$. Observe that $\wedge\sigma$, marked below by $\bullet$, need not necessarily lie in $\sigma$.
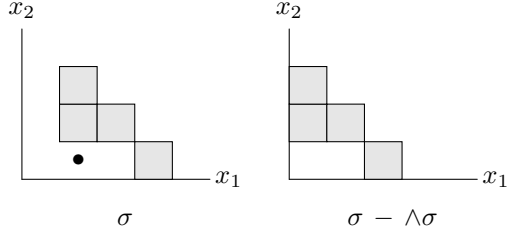
    \begin{center}
        \begin{tikzpicture}[caption=$\sigma$]
            \axestwo{5}{4}
            \emptybox{3}{0}{gray!20!white}
            \emptybox{1}{1}{gray!20!white}
            \emptybox{2}{1}{gray!20!white}
            \emptybox{1}{2}{gray!20!white}
            \node at (0.75, 0.25) {$\bullet$};
        \end{tikzpicture}
        \begin{tikzpicture}[caption=$\sigma - \wedge\sigma$]
            \axestwo{5}{4}
            \emptybox{2}{0}{gray!20!white}
            \emptybox{0}{1}{gray!20!white}
            \emptybox{1}{1}{gray!20!white}
            \emptybox{0}{2}{gray!20!white}
        \end{tikzpicture}
    \end{center}
\end{exa}

\begin{defn}
    Say a skew shape $\sigma$ is an abstract skew shape if $\sigma$ is nonempty and $\sigma = \sigma - \wedge\sigma$.
\end{defn}

 Now, let $I \subseteq K \subseteq S$ and $J \subseteq L \subseteq S$ be monomial ideals of finite colength. Recall that standard shapes are in an inclusion-reversing bijection with finite colength monomial ideals of $S$, and observe that a quotient of finite colength monomial ideals of $S$ then corresponds to the skew shape given by the difference of the standard shapes of the ideals under the bijection.

 We recall the following result from \cite{CSS24}:

 \begin{lem}[{\cite[Lemma 9]{CSS24}}]\label{LEM_BIJ}
     Let $\zeta$ be the corresponding skew shape to $K / I$ and $\xi$ the corresponding skew shape to $L / J$. Then, $S$-module isomorphisms $\phi : K / I \to L / J$ that are monomial maps correspond to bijections $\varphi : \zeta \to \xi$ that act translationally on connected components; that is, if $\sigma \subseteq \zeta$ is a connected component, then there exists some $v \in \bbZ^n$ for which for all $w \in \sigma$, $\varphi(w) = w + v$. In particular, $\sigma$ and $\varphi(\sigma)$ are translationally equivalent. 
 \end{lem}

\begin{rem}
    \Cref{LEM_BIJ} says in particular that if $K / I$ and $L / J$ are isomorphic as $S$-modules through a monomial map, then their corresponding skew shapes $\zeta$ and $\xi$ have the same number of connected components. Moreover, there is an ordering of $\zeta_1, \cdots, \zeta_\ell$ the connected components of $\zeta$ and $\xi_1, \cdots, \xi_\ell$ the connected components of $\xi$ so that $\zeta_1 - \wedge\zeta_1, \cdots, \zeta_\ell - \wedge\zeta_\ell$ and $\xi_1 - \wedge\xi_1, \cdots, \xi_\ell - \wedge\xi_\ell$ are identical sequences.
\end{rem}

\begin{rem}[{\cite[Corollary 10]{CSS24}}]\label{REM_BIJlm}
    Let $\lambda$ be the corresponding standard shape to $S / I$ and $\mu$ the corresponding standard shape to $S / J$. Suppose $\phi : K / I \to L / J$ is a monomial map and an $S$-module isomorphism. Let $\zeta_1, \cdots, \zeta_\ell$ and $\xi_1, \cdots, \xi_\ell$ be as above, ordered so that $\xi_j = \phi(\zeta_j)$ for $j \in [\ell]$. Let $\nu_1, \cdots, \nu_\ell$ be connected abstract skew shapes and $\bm{b}_1, \cdots, \bm{b}_\ell$ and $\bm{c}_1, \cdots, \bm{c}_\ell$ be elements of $\bbN^n$ such that for $j \in [\ell]$, $\nu_j + \bm{b}_j = \zeta_j$ and $\nu_j + \bm{c}_j = \xi_j$; i.e., $\zeta_j - \wedge\zeta_j = \nu_j = \xi_j - \wedge\xi_j$, $\bm{b}_j = \wedge\zeta_j$, and $\bm{c}_j = \wedge\xi_j$. Then, observe that:
    \begin{enumerate}[(i)]
        \item $\nu_1 + \bm{b}_1, \cdots, \nu_\ell + \bm{b}_\ell$ are disjoint, contained in $\lambda$, and closed under $\geq$ in $\lambda$
        \item $\nu_1 + \bm{c}_1, \cdots, \nu_\ell + \bm{c}_\ell$ are disjoint, contained in $\mu$, and closed under $\geq$ in $\mu$
    \end{enumerate}

    In fact, there is a bijective correspondence between choices $(\lambda, \mu, \nu, \bm{b}, \bm{c})$ satisfying these conditions and modules arising from $(I, J, K, L, \phi)$ paired with an ordering on the connected components of $K / I$. In particular, it is well-defined to speak of the module arising from $(\lambda, \mu, \nu, \bm{b}, \bm{c})$, and it follows that every combinatorial module is isomorphic to a module arising from some $(\lambda, \mu, \nu, \bm{b}, \bm{c})$ and vice-versa.

    For such a module, we will refer to $\lambda$ as the left side of the module and $\mu$ as the right side.
\end{rem}

\begin{exa}
    With notation as per the previous remark, let $n = 2$, $K = L = (x_1, x_2)^3$, $I = (x_1^4, x_1^3x_2, x_1^2x_2^2, x_2^4)$, and $J = (x_1^4, x_1^3x_2, x_1x_2^3, x_2^4)$, with corresponding standard shapes as depicted.
    
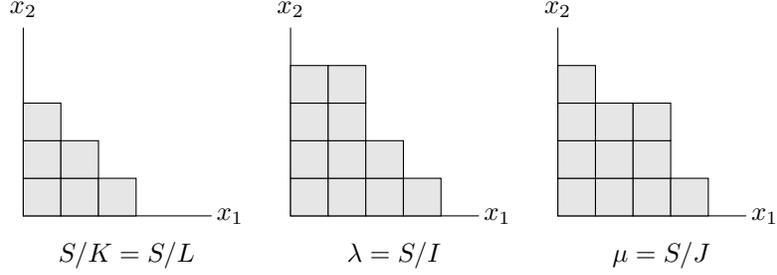
\begin{center}
        \begin{tikzpicture}[caption=${S / K = S / L}$]
            \axestwo{5}{5}
            \tableau{3,2,1}{0}{0}{gray!20!white}
        \end{tikzpicture}
        \begin{tikzpicture}[caption=${\lambda = S / I}$]
            \axestwo{5}{5}
            \tableau{4,4,2,1}{0}{0}{gray!20!white}
        \end{tikzpicture}
        \begin{tikzpicture}[caption=${\mu = S / J}$]
            \axestwo{5}{5}
            \tableau{4,3,3,1}{0}{0}{gray!20!white}
        \end{tikzpicture}
    \end{center}
    We then fix an arbitrary ordering on the connected components of $\zeta = K / I$, which determines an ordering on the connected abstract skew shapes $\nu$.
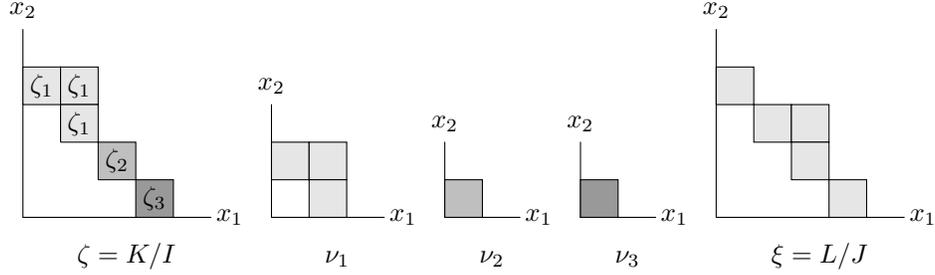
    \begin{center}
        \begin{tikzpicture}[caption=${\zeta = K / I}$]
            \axestwo{5}{5}
            \numberedbox{0}{3}{$\zeta_1$}{gray!20!white}
            \numberedbox{1}{3}{$\zeta_1$}{gray!20!white}
            \numberedbox{1}{2}{$\zeta_1$}{gray!20!white}
            \numberedbox{2}{1}{$\zeta_2$}{gray!50!white}
            \numberedbox{3}{0}{$\zeta_3$}{gray!80!white}
        \end{tikzpicture}
        \hspace{-1em}
        \begin{tikzpicture}[caption=$\nu_1$]
            \axestwo{3}{3}
            \emptybox{0}{1}{gray!20!white}
            \emptybox{1}{1}{gray!20!white}
            \emptybox{1}{0}{gray!20!white}
        \end{tikzpicture}
        \hspace{-1em}
        \begin{tikzpicture}[caption=$\nu_2$]
            \axestwo{2}{2}
            \emptybox{0}{0}{gray!50!white}
        \end{tikzpicture}
        \hspace{-1em}
        \begin{tikzpicture}[caption=$\nu_3$]
            \axestwo{2}{2}
            \emptybox{0}{0}{gray!80!white}
        \end{tikzpicture}
        \hspace{-1em}
        \begin{tikzpicture}[caption=${\xi = L / J}$]
            \axestwo{5}{5}
            \emptybox{0}{3}{gray!20!white}
            \emptybox{1}{2}{gray!20!white}
            \emptybox{2}{2}{gray!20!white}
            \emptybox{2}{1}{gray!20!white}
            \emptybox{3}{0}{gray!20!white}
        \end{tikzpicture}
    \end{center}
    Then, we see that $\bm{b} = ((0, 2), (2, 1), (3, 0))$, and that there are precisely two $S$-module isomorphisms $\phi : K / I \to L / J$ that are also monomial maps, corresponding to $\bm{c} = ((1, 1), (3, 0), (0, 3))$ and $\bm{c} = ((1, 1), (0, 3), (3, 0))$.
    \begin{center}
        \begin{tikzpicture}[caption=${\bm{c} = ((1, 1), (3, 0), (0, 3))}$]
            \axestwo{5}{5}
            \numberedbox{1}{2}{$\xi_1$}{gray!20!white}
            \numberedbox{2}{2}{$\xi_1$}{gray!20!white}
            \numberedbox{2}{1}{$\xi_1$}{gray!20!white}
            \numberedbox{3}{0}{$\xi_2$}{gray!50!white}
            \numberedbox{0}{3}{$\xi_3$}{gray!80!white}
        \end{tikzpicture}
        \hspace{3em}
        \begin{tikzpicture}[caption=${\bm{c} = ((1, 1), (0, 3), (3, 0))}$]
            \axestwo{5}{5}
            \numberedbox{1}{2}{$\xi_1$}{gray!20!white}
            \numberedbox{2}{2}{$\xi_1$}{gray!20!white}
            \numberedbox{2}{1}{$\xi_1$}{gray!20!white}
            \numberedbox{3}{0}{$\xi_3$}{gray!80!white}
            \numberedbox{0}{3}{$\xi_2$}{gray!50!white}
        \end{tikzpicture}
    \end{center}
\end{exa}

Translating \Cref{GQM} into the language of skew shapes, we have:

\begin{prop}[{\cite[Proposition 12]{CSS24}}]
    The module arising from $(\lambda, \mu, \nu, \bm{b}, \bm{c})$ is a counterexample if and only if $|\nu| > |\lambda \cap \mu|$, where if $\nu$ is the sequence $\nu_1, \cdots, \nu_\ell$ then $|\nu|$ denotes the sum $|\nu_1| + \cdots + |\nu_\ell|$.
\end{prop}


\section{Floor plans}

We proceed by reducing the class of combinatorial modules necessary to consider in producing a counterexample. In doing so, we extend the notion of floor plans introduced in \cite{CSS24} to allow for arbitrary sequences of connected abstract skew shapes $\nu$.

\begin{lem}\label{LEM_SCA}
    Suppose the module arising from $(\lambda, \mu, \nu, \bm{b}, \bm{c})$ is a counterexample. Write $\nu = (\nu_1, \cdots, \nu_\ell)$ and let $\lambda'$ be the closure of $(\nu_1 + \bm{b}_1) \sqcup \cdots \sqcup (\nu_\ell + \bm{b}_\ell)$ under $\leq$. Define $\mu'$ symmetrically. Then, the module arising from $(\lambda', \mu', \nu, \bm{b}, \bm{c})$ is also a counterexample.
\end{lem}

\begin{proof}
    Clearly, $\lambda' \subseteq \lambda$ and $\mu' \subseteq \mu$ are standard shapes. As $\nu_1 + \bm{b}_1, \cdots, \nu_\ell + \bm{b}_\ell$ are disjoint and closed under $\geq$ in $\lambda$, and they are contained in $\lambda'$ by construction, they are disjoint and closed under $\geq$ in $\lambda'$. The same statements hold for $\nu_1 + \bm{c}_1, \cdots, \nu_\ell + \bm{c}_\ell$ in $\mu'$. So, by \Cref{REM_BIJlm}, $(\lambda', \mu', \nu, \bm{b}, \bm{c})$ does give rise to a module. Moreover, as $|\nu| > |\lambda \cap \mu| \geq |\lambda' \cap \mu'|$, this module is a counterexample.
\end{proof}

\begin{defn}
    Suppose we are given a sequence $\nu_1, \cdots, \nu_\ell$ of connected abstract skew shapes and $\bm{b}_1, \cdots, \bm{b}_\ell,\allowbreak \bm{c}_1, \cdots, \bm{c}_\ell \in \bbN^n$ for which $\nu_1 + \bm{b}_1, \cdots, \nu_\ell + \bm{b}_\ell$ are disjoint and $\nu_1 + \bm{c}_1, \cdots, \nu_\ell + \bm{c}_\ell$ are disjoint. Denote by $\lambda_{(\nu, \bm{b})}$ the closure of $(\nu_1 + \bm{b}_1) \sqcup \cdots \sqcup (\nu_\ell + \bm{b}_\ell)$ under $\leq$ and $\mu_{(\nu, \bm{c})}$ the closure of $(\nu_1 + \bm{c}_1) \sqcup \cdots \sqcup (\nu_\ell + \bm{c}_\ell)$ under $\leq$. If $\nu_1 + \bm{b}_1, \cdots, \nu_\ell + \bm{b}_\ell$ are closed under $\geq$ in $\lambda_{(\nu, \bm{b})}$ and $\nu_1 + \bm{c}_1, \cdots, \nu_\ell + \bm{c}_\ell$ are closed under $\geq$ in $\mu_{(\nu, \bm{c})}$, then define the module arising from $(\nu, \bm{b}, \bm{c})$ to be the module arising from $(\lambda_{(\nu, \bm{b})}, \mu_{(\nu, \bm{c})}, \nu, \bm{b}, \bm{c})$. Note that this indeed yields a combinatorial module by \Cref{REM_BIJlm}. For $N$ any $S$-module, say $N$ is a scaffolded combinatorial module if $N$ is isomorphic to the module arising from $(\nu, \bm{b}, \bm{c})$.
\end{defn}

\begin{exa}\label{EXA_LPLUSS}
    Suppose $n = 3$ and $\nu$ is given by the following:
    
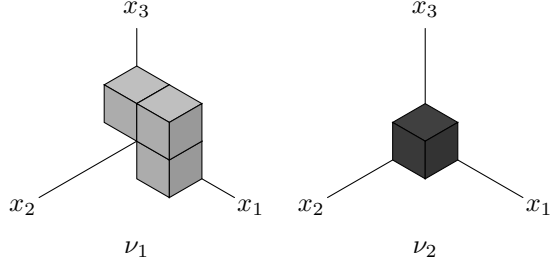
\begin{center}
        \begin{tikzpicture}[caption=$\nu_1$]
            \axes{3}{3}{3}
            \planepartition[1]{{1}}{lightgray}
            \planepartition{{0,2}}{lightgray}
        \end{tikzpicture}
        \begin{tikzpicture}[caption=$\nu_2$]
            \axes{3}{3}{3}
            \planepartition{{1}}{darkgray}
        \end{tikzpicture}
    \end{center}
    Then, with $\bm{b} = ((0, 0, 1), (0, 2, 1))$, we depict a possible choice of $\lambda$ compatible with the choice of $\nu$ and $\bm{b}$, alongside $\lambda_{(\nu, \bm{b})}$. Observe that $\lambda_{(\nu, \bm{b})} \subseteq \lambda$.
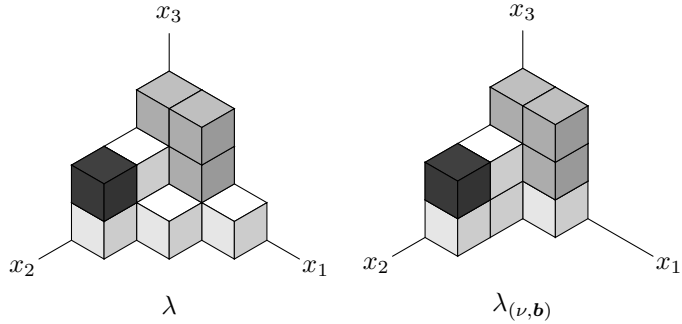
    \begin{center}
        \begin{tikzpicture}[caption=$\lambda$]
            \axes{4}{4}{4}
            \planepartition{{3,3,1},{2,1},{2}}{white}
            \planepartition[2]{{1}}{lightgray}
            \planepartition[1]{{0,2}}{lightgray}
            \planepartition[1]{{},{},{1}}{darkgray}
        \end{tikzpicture}
        \begin{tikzpicture}[caption=$\lambda_{(\nu, \bm{b})}$]
            \axes{4}{4}{4}
            \planepartition{{3,3},{2},{2}}{white}
            \planepartition[2]{{1}}{lightgray}
            \planepartition[1]{{0,2}}{lightgray}
            \planepartition[1]{{},{},{1}}{darkgray}
        \end{tikzpicture}
    \end{center}
\end{exa}

In general, given a module arising from $(\nu, \bm{b}, \bm{c})$, we will assume that $\nu$ is the sequence $\nu_1, \cdots, \nu_\ell$ and we will denote the $\lambda_{(\nu, \bm{b})}$ and $\mu_{(\nu, \bm{c})}$ defined above by just $\lambda$ and $\mu$ when it is clear from context.

By \Cref{LEM_SCA}, it suffices to consider scaffolded combinatorial modules. Fix now $n = 3$. Let $\pi : \bbN^3 \to \bbN^2$ be the projection onto the first two coordinates. We will also implicitly embed $\bbN^2$ into $\bbN^3$ via the first two coordinates, in particular allowing for expressions of the form $\sigma + a$ where $\sigma \subseteq \bbN^3$ and $a \in \bbN^2$.

Suppose $(\nu, \bm{b}, \bm{c})$ were a counterexample. The same argument to the proof of \Cref{LEM_SCA} shows that if $(\nu, \bm{b}', \bm{c}')$ gave rise to a module satisfying $\bm{b}'_j \leq \bm{b}_j$ and $\bm{c'}_j \leq \bm{c}_j$ for all $j \in [\ell]$, then $(\nu, \bm{b}', \bm{c}')$ would also be a counterexample. We introduce a partial order by writing $(\nu, \bm{b}', \bm{c}') \leq (\nu, \bm{b}, \bm{c})$. In particular, given a counterexample $(\nu, \bm{b}, \bm{c})$, we may ask for a minimal counterexample $(\nu, \bm{b}', \bm{c}') \leq (\nu, \bm{b}, \bm{c})$, with the additional constraints that $\pi(\bm{b}'_j) = \pi(\bm{b}_j)$ and $\pi(\bm{c}'_j) = \pi(\bm{c}_j)$ for all $j \in [\ell]$. In this section, we will show that there is a minimum $(\nu, \bm{b}', \bm{c}') \leq (\nu, \bm{b}, \bm{c})$ satisfying these constraints.
\begin{exa}\label{EXA_LPLUSS2}
    Take $\nu$ and $\bm{b}$ as given in \Cref{EXA_LPLUSS}. Then, $\pi(\bm{b}_1) = (0, 0)$ and $\pi(\bm{b}) = (0, 2)$, and it is clear that $\bm{b}' = ((0, 0, 0), (0, 2, 0))$ is the minimum $\bm{b}' \leq \bm{b}$ with $\pi(\bm{b}'_1) = \pi(\bm{b}_1)$ and $\pi(\bm{b}'_2) = \pi(\bm{b}_2)$ for which $\lambda_{(\nu, \bm{b}')}$ satisfies condition (i) of \Cref{REM_BIJlm}.
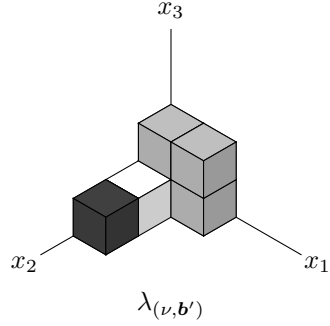
    \begin{center}
        \begin{tikzpicture}[caption=$\lambda_{(\nu, \bm{b}')}$]
            \axes{4}{4}{4}
            \planepartition{{2,2},{1},{1}}{white}
            \planepartition[1]{{1}}{lightgray}
            \planepartition{{0,2}}{lightgray}
            \planepartition{{},{},{1}}{darkgray}
        \end{tikzpicture}
    \end{center}
\end{exa}

Observe first that if $\sigma \subseteq \bbN^3$ is a skew shape, then so too is $\pi(\sigma) \subseteq \bbN^2$.

\begin{lem}\label{LEM_PROJDIS}
    Given a module arising from $(\nu, \bm{b}, \bm{c})$, for $i, j \in [\ell], i \neq j$, $\pi(\nu_i + \bm{b}_i)$ and $\pi(\nu_j + \bm{b}_j)$ are disjoint, and $\pi(\nu_i + \bm{c}_i)$ and $\pi(\nu_j + \bm{c}_j)$ are disjoint.
\end{lem}

\begin{proof}
    By symmetry, it suffices to show that $\pi(\nu_i + \bm{b}_i)$ and $\pi(\nu_j + \bm{b}_j)$ are disjoint. Suppose otherwise, and take $(a_1, a_2) \in \pi(\nu_i + \bm{b}_i) \cap \pi(\nu_j + \bm{b}_j)$. Then, there are $a_3, a_3' \geq 0$ with $(a_1, a_2, a_3) \in \nu_i + \bm{b}_i$ and $(a_1, a_2, a_3') \in \nu_j + \bm{b}_j$. Suppose without loss of generality that $a_3 \leq a_3'$. Then, $(a_1, a_2, a_3') \in \lambda$, and $\nu_i + \bm{b}_i$ is closed under $\geq$ in $\lambda$, so $(a_1, a_2, a_3') \in \nu_i + \bm{b}_i$, contradicting the disjointedness of $\nu_i + \bm{b}_i$ and $\nu_j + \bm{b}_j$.
\end{proof}

\begin{defn}
    A floor plan $(\nu, b, c)$ is a sequence $\nu_1, \cdots, \nu_\ell \subseteq \bbN^3$ of connected abstract skew shapes and sequences $b_1, \cdots, b_\ell, \allowbreak c_1, \cdots, c_\ell \in \bbN^2$ such that for $i, j \in [\ell], i \neq j$, $\pi(\nu_i) + b_i$ and $\pi(\nu_j) + b_j$ are disjoint and $\pi(\nu_i) + c_i$ and $\pi(\nu_j) + c_j$ are disjoint.
\end{defn}

\begin{cor}
    Given a module arising from $(\nu, \bm{b}, \bm{c})$, for $j \in [\ell]$, $\pi(\nu_j + \bm{b}_j) = \pi(\nu_j) + \pi(\bm{b}_j)$ and $\pi(\nu_j + \bm{c}_j) = \pi(\nu_j) + \pi(\bm{c}_j)$. Let $b_j = \pi(\bm{b}_j)$ and $c_j = \pi(\bm{c}_j)$. Then, $(\nu, b, c)$ is a floor plan.
\end{cor}

\begin{defn}
    Given a floor plan $(\nu, b, c)$, say a module arising from $(\nu, \bm{b}, \bm{c})$ is a realization of the floor plan if $b_j = \pi(\bm{b}_j)$ and $c_j = \pi(\bm{c}_j)$ for $j \in [\ell]$.
\end{defn}

Multiple scaffolded combinatorial modules may be realizations of the same floor plan. However, a priori it is not clear that every floor plan is realized by a scaffolded combinatorial module.

\begin{exa}
    Take $\nu$ as given in \Cref{EXA_LPLUSS}. Then, it is possible for $b = ((0, 0), (0, 2))$ to be part of a valid floor plan, but not for $b = ((0, 0), (0, 0))$.
    \begin{center}
        \begin{tikzpicture}[caption=${b = ((0, 0), (0, 2))}$]
            \axes{4}{4}{4}
            \planepartition[1]{{1}}{lightgray}
            \planepartition{{0,2}}{lightgray}
            \planepartition{{},{},{1}}{darkgray}
        \end{tikzpicture}
        \begin{tikzpicture}[caption=${b = ((0, 0), (0, 0))}$]
            \axes{4}{4}{4}
            \planepartition{{1}}{darkgray}
            \planepartition[1]{{1}}{lightgray}
            \planepartition{{0,2}}{lightgray}
        \end{tikzpicture}
    \end{center}
    \Cref{EXA_LPLUSS} and \Cref{EXA_LPLUSS2} show the left sides for possible realizations of a floor plan with $b = ((0, 0), (0, 2))$.
\end{exa}

\begin{defn}
    Recall that any skew shape $\sigma \subseteq \bbN^3$ can be written as $\overline{\sigma} \smallsetminus \underline{\sigma}$ for $\overline{\sigma}, \underline{\sigma}$ standard shapes. Define the upper height function of $\sigma$ to be $\overline{p} : \bbN^2 \to \bbN$, $(a_1, a_2) \mapsto |\{a_3 \geq 0 : (a_1, a_2, a_3) \in \overline{\sigma}\}|$. Similarly, define the lower height function of $\sigma$ to be $\underline{p} : \bbN^2 \to \bbN$, $(a_1, a_2) \mapsto |\{a_3 \geq 0 : (a_1, a_2, a_3) \in \underline{\sigma}\}|$.
\end{defn}

\begin{rem}\label{LEM_PPROPS}
    There is a bijection between abstract skew shapes $\sigma \subseteq \bbN^3$ and pairs of functions $(\overline{p}, \underline{p})$ from $\bbN^2$ to $\bbN$ satisfying:
    \begin{enumerate}[(i)]
        \item $\overline{p} \neq 0$
        \item $\overline{p}$ has finite support
        \item $\overline{p} \geq \underline{p}$
        \item $\overline{p}$ and $\underline{p}$ are nonincreasing
        \item $\overline{p}(a) = \underline{p}(a) > 0 \implies \exists a' > a, \overline{p}(a') = \overline{p}(a) > \underline{p}(a')$
    \end{enumerate}
    Moreover, allowing $\overline{p} = 0$ yields a bijection between the collection of abstract skew shapes in $\bbN^3$ along with $\varnothing$ and pairs of functions from $\bbN^2$ to $\bbN$ satisfying the remaining four properties.
\end{rem}

Given a floor plan $(\nu, b, c)$, we will write $\overline{p}_1, \cdots, \overline{p}_\ell$ and $\underline{p}_1, \cdots, \underline{p}_\ell$ for the upper and lower height functions of $\nu_1, \cdots, \nu_\ell$ respectively.

\begin{rem}\label{LEM_USEPS}
    Let $\sigma \subseteq \bbN^3$ be an abstract skew shape with height functions $\overline{p}, \underline{p}$. Let $\bm{a} = (a_1, a_2, a_3) \in \bbN^3$ and let $a = (a_1, a_2)$. Then, $\bm{a} \in \sigma$ if and only if $\underline{p}(a) \leq a_3 < \overline{p}(a)$.
\end{rem}


\begin{defn}\label{DEF_main_partial_order}
    Given a floor plan $(\nu, b, c)$, we define a relation $\leq_b$ on $\nu$ as follows: for $i, j \in [\ell]$, say that $\nu_i \leq_b \nu_{j}$ if there exist $v \in \pi(\nu_i) + b_i$ and $w \in \pi(\nu_{j}) + b_{j}$ for which $v \leq w$. Define $\leq_c$ analogously.
\end{defn}

\begin{rem}
    Given a floor plan $(\nu, b, c)$, we can make $\leq_b$ and $\leq_c$ into partial orders by taking their transitive closures and proving that the resulting relation is antisymmetric, justifying the choice of notation. We will not do so here. However, we will show that $\leq_b$ and $\leq_c$ are in fact antisymmetric prior to taking transitive closures in the following series of lemmas.
\end{rem}

\begin{defn}
    Given $v_0, \cdots, v_p$ a path in $\bbN^2$, say the path is northeast if for all $j \in [p]$, $v_j - v_{j - 1} \in \{e_1, e_2\}$. Define northwest, southeast, and southwest analogously. We say a path is unidirectional if it is northeast, northwest, southeast, or southwest.
\end{defn}

\begin{rem}\label{REM_PATHANNOYING}
    If $v_0 = (v_{0, 1}, v_{0, 2}), \cdots, v_p = (v_{p, 1}, v_{p, 2}) \in \bbN^2$ is a path, then it passes through every horizontal coordinate between $v_{0, 1}$ and $v_{p, 1}$ and every vertical coordinate between $v_{0, 2}$ and $v_{p, 2}$. Moreover, if it is a northeast path, then $v_{0, 1}, \cdots, v_{p, 1}$ and $v_{0, 2}, \cdots, v_{p, 2}$ are weakly increasing sequences. Analogous statements hold for northwest, southeast, and southwest paths.
\end{rem}

\begin{defn}
    Given $r = (r_1, r_2), s = (s_1, s_2) \in \bbN^2$, by the rectangle defined by $r$ and $s$ we mean the subset $$\{(v_1, v_2) \in \bbN^2: \min\{r_1, s_1\} \leq v_1 \leq \max\{r_1, s_1\}, \min\{r_2, s_2\} \leq v_2 \leq \max\{r_2, s_2\}\}.$$ In particular, when $r \leq s$, the rectangle defined by $r$ and $s$ is $\{v \in \bbN^2 : r \leq v \leq s\}$.
\end{defn}

\begin{lem}\label{LEM_CONTAINRECTS}
    Given a skew shape $\sigma$ and $r, s \in \sigma$, $r \leq s$, we have that $\sigma$ contains the rectangle defined by $r$ and $s$.
\end{lem}

\begin{proof}
    Take $r \leq v \leq s$. Clearly, $v \in \overline{\sigma}$. Now, if $v \in \underline{\sigma}$, then $r \in \underline{\sigma}$ as $\underline{\sigma}$ is closed under $\leq$, so $r \not\in \sigma$, a contradiction. Hence, $v \in \overline{\sigma} \smallsetminus \underline{\sigma} = \sigma$.
\end{proof}

\begin{lem}\label{LEM_PATHINRECT}
    Given $\sigma \subseteq \bbN^2$ a connected skew shape and $r = (r_1, r_2), s = (s_1, s_2) \in \sigma$, there exists a unidirectional path from $r$ to $s$ in $\sigma$ within the rectangle defined by $r$ and $s$.
\end{lem}

\begin{proof}
    Assume that $r_1 \leq s_1$ and $r_2 \leq s_2$. The other three cases are proven symmetrically. As $\sigma$ is connected, by definition there is a path of minimal length $r = v_0, \cdots, v_p = s$ from $r$ to $s$ in $\sigma$. We induct on $p$.

    If $p = 0$, then $r = s$ and we are done. Inductively, suppose that $p > 0$ and that for all $r'$ and $s'$ in $\sigma$, if $r_1' \leq s_1'$ and $r_2' \leq s_2'$ and there exists a path of minimal length from $r'$ to $s'$ in $\sigma$ of length less than $p$, then that path is within the rectangle defined by $r'$ and $s'$. Now, suppose that $j \in [p]$ is the first index at which $v_j$ is not within the rectangle defined by $r$ and $s$. If $j > 1$, then the path $v_{j - 1}, \cdots, v_p$ from $v_{j - 1}$ to $v_p$ in $\sigma$ is of length $p - (j - 1) < p$. Consider a path $v_{j - 1} = w_0, \cdots, w_q = v_p$ from $v_{j - 1}$ to $v_p$ in $\sigma$ of minimal length $q$. Then, $q \leq p - (j - 1)$, so by induction this path is within the rectangle defined by $v_{j - 1}$ and $v_p = s$, and hence within the rectangle defined by $r$ and $s$. Then, $v_0, \cdots, v_{j - 2}, w_0, \cdots, w_q$ is a path from $r$ to $s$ in $\sigma$ within the rectangle defined by $r$ and $s$ of length $j - 2 + 1 + q \leq p$.

    So, suppose $j = 1$. Then, as $v_1$ is not within the rectangle defined by $v_0 = r$ and $s$, without loss of generality we may assume that $v_1 - v_0 = -e_1$. In other words, assume $v_1 = (r_1 - 1, r_2)$. As $v_p = s$ and $s_1 > r_1 - 1$, by \Cref{REM_PATHANNOYING}, the path from $v_1$ to $s$ must pass through the horizontal coordinate $r_1$ again, and so there exists $j > 1$ for which $v_j = (r_1, v_{j, 2})$ for some natural number $v_{j, 2}$. But evidently the straight line path $v_0$ to $v_j$ holding the horizontal coordinate $r_1$ constant and varying the vertical constant from $r_2$ to $v_{j, 2}$, contained in $\sigma$ by \Cref{LEM_CONTAINRECTS}, must have strictly smaller length than $v_0, \cdots, v_j$, a contradiction.

    Thus, we have that a path of minimal length from $r$ to $s$ in $\sigma$ is within the rectangle defined by $r$ and $s$. A similar argument to the above shows that the path must in fact be northeast, and we are done.
\end{proof}

\begin{lem}\label{LEM_PATHSCROSS}
    Let $r = (r_1, r_2), s = (s_1, s_2) \in \bbN^2$, $r \leq s$ and let $R$ be the rectangle defined by $r$ and $s$. Let $v_0 = (v_{0, 1}, v_{0, 2}), \cdots, v_p = (v_{p, 1}, v_{p, 2}) \in R$ be a unidirectional path with $v_{0, 1} = r_1$, $v_{p, 1} = s_1$; that is, a unidirectional path from the left edge of $R$ to the right edge of $R$. Let $w_0 = (w_{0, 1}, w_{0, 2}), \cdots, w_q = (w_{q, 1}, w_{q, 2}) \in R$ be a path with $w_{0, 2} = r_2$, $w_{q, 2} = s_2$; that is, a path from the bottom edge of $R$ to the top edge of $R$. Then, $$\{v_0, \cdots, v_p\} \cap \{w_0, \cdots, w_q\} \neq \varnothing.$$
\end{lem}

\begin{proof}
    Suppose the two paths do not intersect. Without loss of generality, suppose $v_0, \cdots, v_p$ is northeast. For all $0 \leq i \leq p$, if $v_{i, 1} = w_{0, 1}$, then $v_{i, 2} > w_{0, 2}$, as $v_{i, 2} \geq w_{0, 2} = r_2$ by construction and the paths would intersect if we had equality.

    Inductively, suppose that for some $j \in [q]$, the following property held: for all $0 \leq i \leq p$, if $v_{i, 1} = w_{j - 1, 1}$, then $v_{i, 2} > w_{j - 1, 2}$. Consider first the case where $w_j - w_{j - 1} = e_2$. If $v_{i, 1} = w_{j, 1}$, then $v_{i, 1} = w_{j - 1, 1}$, so $v_{i, 2} \geq w_{j - 1, 2} + 1 = w_{j, 2}$, where the paths would intersect if we had equality, and so $v_{i, 2} > w_{j, 2}$. Next, the case where $w_j - w_{j - 1} = -e_2$ is trivial.

    Now, if $w_j - w_{j - 1} = e_1$ and for some $0 \leq i \leq p$, $v_{i, 1} = w_{j, 1}$ and $v_{i, 2} \leq w_{j, 2}$, then $i > 0$, as $v_{i, 1} = w_{j - 1, 1} + 1 \geq r_1 + 1 > v_{0, 1}$. By \Cref{REM_PATHANNOYING}, there is some $0 \leq i' < i$ with $v_{i', 1} = w_{j - 1, 1}$, so by the inductive hypothesis, $v_{i', 2} > w_{j - 1, 2}$. But, as $v_0, \cdots, v_p$ is northeast, $v_{i, 2} \geq v_{i', 2} > w_{j - 1, 2} = w_{j, 2}$, a contradiction.

    Finally, if $w_j - w_{j - 1} = -e_1$ and for some $0 \leq i \leq p$, $v_{i, 1} = w_{j, 1}$ and $v_{i, 2} \leq w_{j, 2}$, then in a similar fashion we conclude that $i < p$ and that there is some $i < i' \leq p$ with $v_{i', 1} = w_{j - 1, 1}$. Take the least such $i'$. Then, by \Cref{REM_PATHANNOYING}, as $v_0, \cdots, v_p$ is northeast, $v_{i' - 1, 1} \leq v_{i', 1}$, where the inequality is strict by minimality of $i'$, so $v_{i' - 1, 1} = v_{i', 1} - 1 = w_{j, 1}$. In other words, $v_{i'} - v_{i' - 1} = e_1$. It follows that $v_{i' - 1, 2} = v_{i', 2} > w_{j - 1, 2} = w_{j, 2}$. Hence, $v_i$ and $v_{i' - 1}$ are in the same column as $w_j$, below and above respectively, and so being a northeast path, $v_0, \cdots, v_p$ must contain $w_j$, a contradiction.
\end{proof}

\begin{lem}
    Given a floor plan $(\nu, b, c)$, for $i, j \in [\ell]$, if $\nu_i \leq_b \nu_j$ and $\nu_j \leq_b \nu_i$, then $i = j$. The analogous result holds for $\leq_c$.
\end{lem}

\begin{proof}
    Suppose otherwise. Take $i, j \in [\ell], i \neq j$ with $\nu_i <_b \nu_j$ and $\nu_j <_b \nu_i$. Suppose these are inequalities are witnessed by $v, v' \in \pi(\nu_i) + b_i$, $w, w' \in \pi(\nu_j) + b_j$, $v \leq w$, $w' \leq v'$. Note that as $i \neq j$, $\pi(\nu_i) + b_i$ and $\pi(\nu_j) + b_j$ are disjoint, and so $v < w$ and $w' < v'$.

    Now, $w \not\leq w'$, as otherwise $v < w \leq w' < v'$, implying $w, w' \in \pi(\nu_i) + b_i$ \Cref{LEM_CONTAINRECTS}. Similarly, $v' \not\leq v$. Write $w = (w_1, w_2)$, $w' = (w'_1, w'_2)$, $v = (v_1, v_2)$, and $v' = (v'_1, v'_2)$. 
    
    Suppose $w_1 \leq w'_1$. Then, by incomparability, $w_2 > w'_2$. By definition, $w_2 > v_2$. And, $w_2 > v'_2$, as otherwise $v'_2 \geq w_2$ and $v'_1 \geq w'_1 \geq w_1$ would imply that $v' \geq w' > v$, so by \Cref{LEM_CONTAINRECTS}, $w \in \pi(\nu_i) + b_i$, a contradiction. Proceeding similarly, one can easily verify that: \begin{align*}
        w_2 &\geq \max\{v_2, v'_2, w_2, w'_2\}, \\
        w'_2 &\leq \min\{v_2, v'_2, w_2, w'_2\}, \\
        v_1 &\leq \min\{v_1, v'_1, w_1, w'_1\}, \\
        \text{and } v'_1 &\geq \max\{v_1, v'_1, w_1, w'_1\}. \\
    \end{align*} In particular, $v, v', w, w'$ are contained within the rectangle $R$ defined by $r = (v_1, w'_2)$, $s = (v'_1, w_2)$, $r \leq s$.
    
    As $\nu_i$ and $\nu_j$ are connected, so too are $\pi(\nu_i) + b_i$ and $\pi(\nu_j) + b_j$. By \Cref{LEM_PATHINRECT}, there exists a unidirectional path $v = v_0, \cdots, v_p = v'$ from $v$ to $v'$ in $\pi(\nu_i) + b_i$ within the rectangle defined by $v$ and $v'$. Observe that this path is contained in $R$ from the left edge of $R$ to the right edge of $R$. Similarly, there exists a unidirectional path $w' = w_0, \cdots, w_p = w$ from $w'$ to $w$ in $\pi(\nu_j) + b_j$ within the rectangle defined by $w$ and $w'$. Observe too that this path is contained in $R$ from the bottom edge of $R$ to the top edge of $R$. But, by \Cref{LEM_PATHSCROSS}, the paths intersect, which is to say that $\pi(\nu_i) + b_i$ and $\pi(\nu_j) + b_j$ intersect, a contradiction.

    The case where $w_1 \geq w'_1$ is proven symmetrically, and we are done.
\end{proof}

\begin{cor}
    Given a floor plan $(\nu, b, c)$, all chains in $\leq_b$ are finite, where by chain we mean a sequence $\nu_{i_0} <_b \cdots <_b \nu_{i_m}$ for $i_0, \cdots, i_m \in [\ell]$. In particular, there are only finitely many chains in $\leq_b$. The analogous results hold for $\leq_c$.
\end{cor}

\begin{defn}
    Given a floor plan $(\nu, b, c)$ and $i, j \in [\ell]$ with $\nu_i <_b \nu_j$, define the $b$-height of $\nu_i$ with respect to $\nu_j$ to be $$h_b(i, j) = \max_{\substack{v \in \pi(\nu_i) + b_i \\ w \in \pi(\nu_{j}) + b_j \\ v \leq w}} (\overline{p}_j(w - b_j) - \underline{p}_i(v - b_i)).$$ Define $h_c$ analogously.
\end{defn}

\begin{rem}
    Given a floor plan $(\nu, b, c)$ and $i, j \in [\ell]$ with $\nu_i <_b \nu_j$, we will see that if $(\nu, \bm{b}, \bm{c})$ is a realization of $(\nu, b, c)$, then $\bm{b}_{i, 3} - \bm{b}_{j, 3} \geq h_b(i, j)$.
\end{rem}

\begin{defn}
    Given a floor plan $(\nu, b, c)$, for $j \in [\ell]$, define the $b$-height of $\nu_j$ to be $$h_b(j) = \max_{\substack{i_0, \cdots, i_m \in [\ell] \\ \nu_{i_0} <_b \cdots <_b \nu_{i_m} \\ i_0 = j}}\sum_{r = 1}^m h_b(i_{r - 1}, i_r).$$ Let $\bm{b}_j \in \bbN^3$ be given by $\bm{b}_j = (b_j, h_b(j))$. Define $h_c$ and $\bm{c}$ analogously. Call $(\nu, \bm{b}, \bm{c})$ the canonical realization of $(\nu, b, c)$.
\end{defn}

\begin{exa}
    Suppose $\nu$ is given by the following:
    
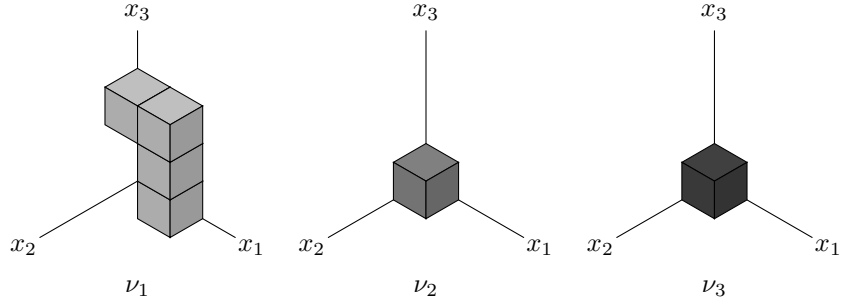
\begin{center}
        \begin{tikzpicture}[caption=$\nu_1$]
            \axes{3}{3}{4}
            \planepartition[2]{{1}}{lightgray}
            \planepartition{{0,3}}{lightgray}
        \end{tikzpicture}
        \begin{tikzpicture}[caption=$\nu_2$]
            \axes{3}{3}{4}
            \planepartition{{1}}{gray}
        \end{tikzpicture}
        \begin{tikzpicture}[caption=$\nu_3$]
            \axes{3}{3}{4}
            \planepartition{{1}}{darkgray}
        \end{tikzpicture}
    \end{center}
    Then, $b = ((0, 0), (0, 2), (0, 3))$ appears as follows:
    
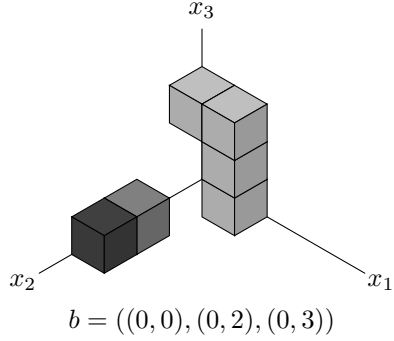
\begin{center}
        \begin{tikzpicture}[caption=${b = ((0, 0), (0, 2), (0, 3))}$]
            \axes{5}{5}{4}
            \planepartition[2]{{1}}{lightgray}
            \planepartition{{0,3}}{lightgray}
            \planepartition{{},{},{1}}{gray}
            \planepartition{{},{},{},{1}}{darkgray}
        \end{tikzpicture}
    \end{center}
    We see that $\nu_1 <_b \nu_2 <_b \nu_3$. Now, $h_b(2, 3) = 1$ is witnessed by $(0, 2) \in \pi(\nu_2) + b_2$, $(0, 3) \in \pi(\nu_3) + b_3$, $(0, 2) \leq (0, 3)$, $$\overline{p}_3((0, 3) - b_3) - \underline{p}_2((0, 2) - b_2) = \overline{p}_3(0, 0) - \underline{p}_2(0, 0) = 1 - 0 = 1.$$ However, observe $h_b(1, 2) = -1$: indeed, the only $v \in \pi(\nu_1) + b_1$ and $w \in \pi(\nu_2) + b_2$ with $v \leq w$ is $v = (0, 0)$ and $w = (0, 2)$, and we verify that $$\overline{p}_2((0, 2) - b_2) - \underline{p}_1((0, 0) - b_1) = \overline{p}_2(0, 0) - \underline{p}_1(0, 0) = 1 - 2 = -1.$$ 

    Now, $h_b(3) = 0$, as the only increasing chain in $\leq_b$ starting from $\nu_3$ is the trivial chain $\nu_3$, which results in the empty sum. $h_b(2) = 1$ is witnessed by the chain $\nu_1 <_b \nu_2$, and finally $h_b(1) = 0$ can be witnessed by either the trivial chain $\nu_1$ yielding the empty sum or $\nu_1 <_b \nu_2 <_b \nu_3$ yielding $h_b(1, 2) + h_b(2, 3) = -1 + 1 = 0$.

    Defining $\bm{b}_j = (b_j, h_b(j))$ for $1 \leq j \leq 3$ per the definition of the canonical realization, we have $\bm{b} = ((0, 0, 0), (0, 2, 1), (0, 3, 0))$, yielding:
    
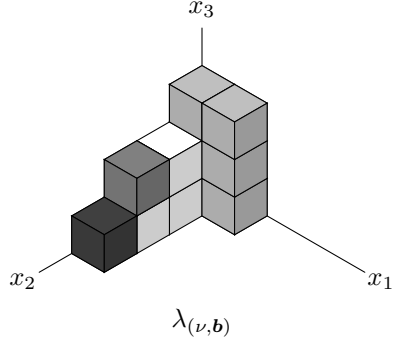
\begin{center}
        \begin{tikzpicture}[caption=$\lambda_{(\nu, \bm{b})}$]
            \axes{5}{5}{4}
            \planepartition{{3,3},{2},{2},{1}}{white}
            \planepartition[2]{{1}}{lightgray}
            \planepartition{{0,3}}{lightgray}
            \planepartition[1]{{},{},{1}}{gray}
            \planepartition{{},{},{},{1}}{darkgray}
        \end{tikzpicture}
    \end{center}
    Visually, it is clear that for this particular choice of $\nu$ and $b$, the $\lambda_{(\nu, \bm{b})}$ we have constructed is the minimum realization.
\end{exa}

\begin{thm}
    Given a floor plan $(\nu, b, c)$, its canonical realization $(\nu, \bm{b}, \bm{c})$ is a realization of the floor plan, and for all other realizations $(\nu, \bm{b}', \bm{c}')$ of the floor plan, $(\nu, \bm{b}, \bm{c}) \leq (\nu, \bm{b}', \bm{c}')$; that is, $\bm{b} \leq \bm{b}'$ and $\bm{c} \leq \bm{c}'$.
\end{thm}

\begin{proof}
    Suppose $(\nu, \bm{b}', \bm{c}')$ gives rise to a module realizing $(\nu, b, c)$. Suppose too that we have $i, j \in [\ell]$, $\nu_i <_b \nu_j$. We claim first that $\bm{b}'_{i, 3} - \bm{b}'_{j, 3} \geq h_b(i, j)$. Otherwise, if $\bm{b}'_{i, 3} - \bm{b}'_{j, 3} < h_b(i, j)$, then by definition of $h_b(i, j)$, there exist $v \in \pi(\nu_i) + b_i$ and $w \in \pi(\nu_j) + b_j$ with $v \leq w$ for which $$\bm{b}'_{i, 3} - \bm{b}'_{j, 3} < \overline{p}_j(w - b_j) - \underline{p}_i(v - b_i).$$ Now, $v - b_i \in \pi(\nu_i)$, which implies $\overline{p}_i(v - b_i) > \underline{p}_i(v - b_i)$, meaning that by \Cref{LEM_USEPS}, $$(v - b_i, \underline{p}_i(v - b_i)) \in \nu_i.$$ Likewise, $w - b_j \in \pi(\nu_j)$, which means that $\overline{p}_j(w - b_j) > \underline{p}_j(w - b_j) \geq 0$, so $$(w - b_j, \overline{p}_i(w - b_j) - 1) \in \nu_j.$$ Shifting back, we have that \begin{align*}
        (v - b_i, \underline{p}_i(v - b_i)) + \bm{b}'_i &= (v - b_i, \underline{p}_i(v - b_i)) + (b_i, \bm{b}'_{i, 3}) \\
        &= (v, \underline{p}_i(v - b_i) + \bm{b}'_{i, 3}) \\
        &\in \nu_i + \bm{b}'_i
    \end{align*} and similarly $(w, \overline{p}_j(w - b_j) + \bm{b}'_{j, 3} - 1) \in \nu_j + \bm{b}'_j$. Note that $w \geq v$ and $$\overline{p}_j(w - b_j) + \bm{b}'_{j, 3} - 1 \geq \underline{p}_i(v - b_i) + \bm{b}'_{i, 3}.$$ With $\lambda' = \lambda_{(\nu, \bm{b}')}$, recall that as $(\nu, \bm{b}', \bm{c}')$ gives rise to a module, $\nu_i + \bm{b}'_i$ is closed under $\geq$ in $\lambda'$. However, this implies that $\nu_i + \bm{b}'_i$ and $\nu_j + \bm{b}'_j$ are not disjoint, a contradiction, and hence we have our claim.
    
    Next, we claim that for all $j \in [\ell]$, $\bm{b}_j \leq \bm{b}'_j$. Indeed, suppose $\bm{b}'_j < \bm{b}_j$. As $\pi(\bm{b}_j) = b_j = \pi(\bm{b}'_j)$, the difference must be in the third coordinate, which is to say that $\bm{b}'_{j, 3} < \bm{b}_{j, 3} = h_b(j)$. By definition of $h_b(j)$, it follows that either $h_b(j) = 0$, in which case $\bm{b}'_{j, 3} < 0$ contradicts $(\nu, \bm{b}', \bm{c}')$ giving rise to a module, or there is some chain $i_0, \cdots, i_m \in [\ell]$, $\nu_{i_0} <_b \cdots <_b \nu_{i_m}$, $i_0 = j$ in $\leq_b$ starting at $\nu_j$ for which $$\bm{b}'_{j, 3} < \sum_{r = 1}^m h_b(i_{r - 1}, i_r).$$ By the first claim, we then have that $$\bm{b}'_{j, 3} < \sum_{r = 1}^m (\bm{b}'_{i_{r - 1}, 3} - \bm{b}'_{i_r, 3}) = \bm{b}'_{i_0, 3} - \bm{b}'_{i_m, 3}.$$ This yields $\bm{b}'_{i_m, 3} < 0$, a contradiction again.

    Repeating the above two claims symmetrically yields $\bm{c} \leq \bm{c}'$, and so we have that $(\nu, \bm{b}, \bm{c}) \leq (\nu, \bm{b}', \bm{c}')$, for all realizations $(\nu, \bm{b}', \bm{c}')$, as desired. It remains to show that $(\nu, \bm{b}, \bm{c})$ gives rise to a module.

    Let $\lambda$ be the closure of $(\nu_1 + \bm{b}_1) \cup \cdots \cup (\nu_\ell + \bm{b}_\ell)$ under $\leq$. As $\pi(\nu_1) + b_1,  \cdots, \allowbreak\pi(\nu_\ell) + b_\ell$ are disjoint, so too are $\nu_1 + \bm{b}_1, \cdots, \nu_\ell + \bm{b}_\ell$ the $\nu_i + b_i$ for $i \in [\ell]$ are each closed under $\geq$ in $\lambda$. Suppose otherwise. Then, there is some $i \in [\ell]$ with $\bm{v} \in \nu_i + \bm{b}_i$, $\bm{u} \in \lambda \smallsetminus (\nu_i + \bm{b}_i)$, $\bm{u} \geq \bm{v}$. Now, by construction of $\lambda$, $\bm{u}$ must come from the closure of some $\nu_j + \bm{b}_j$ under $\leq$ for $j \in [\ell]$, $j \neq i$. Then, there is $\bm{w} \in \nu_j + \bm{b}_j$, $\bm{w} \geq \bm{u} \geq \bm{v}$. Let $v = \pi(\bm{v})$ and $w = \pi(\bm{w})$. Observe that $v \leq w$ witnesses $\nu_i <_b \nu_j$, and that \begin{align*}
        h_b(i) - h_b(j) &= \max_{\substack{i_0, \cdots, i_m \in [\ell] \\ \nu_{i_0} <_b \cdots <_b \nu_{i_m} \\ i_0 = i}}\sum_{r = 1}^m h_b(i_{r - 1}, i_r) - \max_{\substack{i_0, \cdots, i_m \in [\ell] \\ \nu_{i_0} <_b \cdots <_b \nu_{i_m} \\ i_0 = j}}\sum_{r = 1}^m h_b(i_{r - 1}, i_r) \\
        &\geq \max_{\substack{i_0, i_1, \cdots, i_m \in [\ell] \\ \nu_{i_0} <_b \nu_{i_1} <_b \cdots <_b \nu_{i_m} \\ i_0 = i, i_1 = j}}\sum_{r = 1}^m h_b(i_{r - 1}, i_r) - \max_{\substack{i_0, \cdots, i_m \in [\ell] \\ \nu_{i_0} <_b \cdots <_b \nu_{i_m} \\ i_0 = j}}\sum_{r = 1}^m h_b(i_{r - 1}, i_r),
     \tag*{as every chain starting at $\nu_j$ can be extended to one starting at $\nu_i$} \\
     &\geq h_b(i, j).
    \end{align*}
    Without loss of generality, we may assume $\bm{v} = (v, \underline{p}_i(v - b_i) + \bm{b}_{i, 3})$ and $\bm{w} = (w, \overline{p}_j(w - b_j) + \bm{b}_{j, 3} - 1)$; that is, that $\bm{v}$ is the smallest element of $\nu_i + \bm{b}_i$ whose projection is $v$ and $\bm{w}$ is the largest element of $\nu_j + \bm{b}_j$ whose projection is $w$. Then, \begin{align*}
        \bm{w} \geq \bm{v} &\implies \overline{p}_j(w - b_j) + \bm{b}_{j, 3} - 1 \geq \underline{p}_i(v - b_i) + \bm{b}_{i, 3} \\
        &\implies \overline{p}_j(w - b_j) - \underline{p}_i(v - b_i) > h_b(i) - h_b(j) \geq h_b(i, j),
    \end{align*} a contradiction. Hence, the canonical realization is indeed a module and we are done.
\end{proof}

\begin{rem}\label{REM_scaffoldedsuffices}
    Let the scaffolded combinatorial module arising from $(\nu, \bm{b}, \bm{c})$ be a counterexample. By the discussion following \Cref{LEM_SCA}, we may assume that $(\nu, \bm{b}, \bm{c}) \leq (\nu, \bm{b}', \bm{c}')$ for all other counterexamples $(\nu, \bm{b}', \bm{c}')$ with $\pi(\bm{b}_j) = \pi(\bm{b}'_j)$ and $\pi(\bm{c}_j) = \pi(\bm{c}'_j)$ for all $j \in [\ell]$. Define $b_j = \pi(\bm{b}_j)$ and $c_j = \pi(\bm{c}_j)$ for all $j \in [\ell]$. Then, $(\nu, \bm{b}, \bm{c})$ is the canonical realization of $(\nu, b, c)$. In other words, henceforth it suffices to only consider canonical realizations of floor plans.
    
    We will say a floor plan is a counterexample if its canonical realization is a counterexample.
\end{rem}

\section{Bottom slices}
In this section, we will define what it means to take the bottom slice reduction of a floor plan $(\nu, b, c)$. We will prove that the bottom slice reduction of a counterexample remains a counterexample.

When it is clear from context, given a floor plan $(\nu, b, c)$, we will write $\lambda$ and $\mu$ for $\lambda_{(\nu, \bm{b})}$ and $\mu_{(\nu, \bm{c})}$, where $(\nu, \bm{b}, \bm{c})$ is the canonical realization of $(\nu, b, c)$. Being skew shapes, $\lambda$ and $\mu$ have their own upper and lower height functions, which we will denote $\overline{p}_\lambda, \overline{p}_\mu, \underline{p}_\lambda, \underline{p}_\mu$. As $\lambda$ and $\mu$ are in fact standard shapes, their lower height functions are uninteresting.

\begin{lem}\label{LEM_HTEXP}
    Given a floor plan $(\nu, b, c)$, $$\overline{p}_\lambda(a) = \max_{\substack{j \in [\ell] \\ v \in \pi(\nu_j) + b_j \\ a \leq v}} (h_b(j) + \overline{p}_j(v - b_j)),$$ and analogously for $\overline{p}_\mu$.
\end{lem}

\begin{proof}
    We prove the equality for just $\overline{p}_\lambda$. Let $a = (a_1, a_2) \in \bbN^2$. Recall that $\overline{p}_\lambda(a) = |\{a_3 \geq 0 : (a_1, a_2, a_3) \in \lambda\}|$. Now, suppose $\bm{a} = (a_1, a_2, a_3) \in \lambda$. By construction of $\lambda$, there is some $j \in [\ell]$ and $\bm{v} \in \nu_j + \bm{b}_j$ for which $\bm{a} \leq \bm{v}$. Then, writing $\bm{v} = (v_1, v_2, v_3)$ and $v = (v_1, v_2)$, observe that the largest element of $\lambda$ with projection $v$ must be in $\nu_j + \bm{b}_j$, as $\nu_j + \bm{b}_j$ is closed under $\geq$ in $\lambda$ and contains $\bm{v}$. The largest element of $\nu_j + \bm{b}_j$ with projection $v$ is a shift of the largest element of $\nu_j$ with projection $v - b_j$ by $\bm{b}_j$. Recall $\bm{b}_j = (b_j, h_b(j))$ by definition of the canonical realization. Then, $$(v - b_j, \overline{p}_j(v - b_j)) + \bm{b}_j = (v, \overline{p}_j(v - b_j) + h_b(j)),$$ and so $\overline{p}_\lambda(a) \geq \overline{p}_\lambda(v) = h_b(j) + \overline{p}_j(v - b_j)$. Hence, we have one direction of the desired equality.

    To finish, we notice that if $\overline{p}_\lambda(a)$ is strictly greater than the expression, then $(a, \overline{p}_\lambda(a) - 1) \in \lambda$ satisfies \begin{align*}(a, \overline{p}_\lambda(a) - 1) &\not\leq (v, h_b(j) + \overline{p}_j(v - b_j)) \text{ for all $j \in [\ell]$ and $v \in \pi(\nu_j) + b_j$} \\ &= (w + b_j, \overline{p}_j(w) + h_b(j)) \text{ for all $j \in [\ell]$ and $w \in \pi(\nu_j)$} \\ &= w + \bm{b}_j \text{ for all $j \in [\ell]$ and $w \in \pi(\nu_j)$},\end{align*} contradicting $\lambda$ being the closure of $\nu_1 + \bm{b}_1, \cdots, \nu_\ell + \bm{b}_\ell$ under $\leq$.
\end{proof}

\begin{defn}\label{DEF_SLICING_a}
    Given a connected abstract skew shape $\sigma \subseteq \bbN^3$, call $$\sigma^\circ = \{(a_1, a_2, a_3) \in \sigma : a_3 = 0\}$$ the bottom slice of $\sigma$. Observe that $\sigma \smallsetminus \sigma^\circ$ is a skew shape. Let $\Sigma_\sigma$ be a set $$\{\tau - \wedge \tau : \tau \text{ is a connected component of } \sigma \smallsetminus \sigma^\circ\}$$ of connected abstract skew shapes. We say that removing the bottom slice from $\sigma$ yields $\Sigma_\sigma$.
\end{defn}

\begin{rem}\label{REM_HTOFCOM}
    For $\sigma \subseteq \bbN^3$ an abstract skew shape and $\tau$ a connected component of $\sigma \smallsetminus \sigma^\circ$, let $\overline{p}, \underline{p}$ be the height functions of $\sigma$ and $\overline{p}', \underline{p}'$ the height functions of $\tau - \wedge\tau$. Then, for $a \in \pi(\tau)$, $\overline{p}'(a) = \overline{p}(a + \pi(\wedge\tau)) - 1$ and $\underline{p}'(a) \geq \underline{p}(a + \pi(\wedge\tau)) - 1$.
\end{rem}

\begin{exa}
    Take $\sigma = \overline{\sigma} \smallsetminus \underline{\sigma}$ to be the following connected abstract skew shape:
    
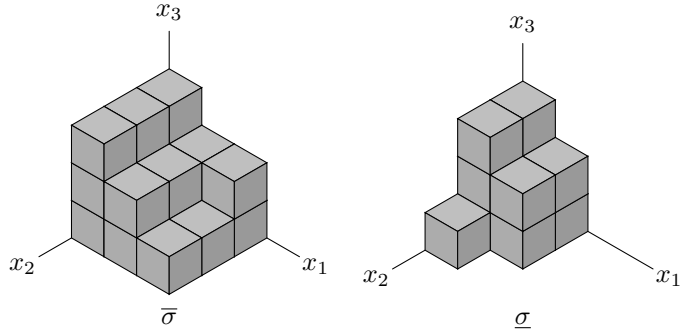
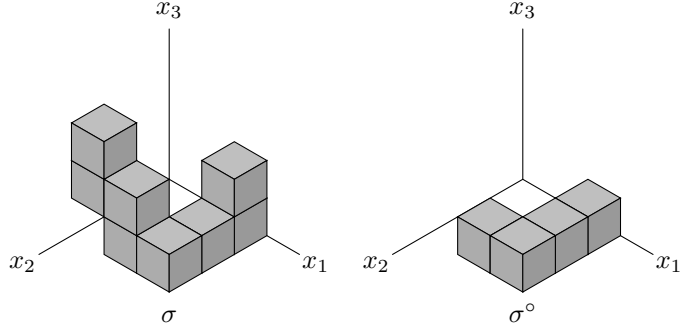
\begin{center}
        \begin{tikzpicture}[caption=$\overline{\sigma}$]
            \axes{4}{4}{4}
            \planepartition{{3,2,2},{3,2,1},{3,2,1}}{lightgray}
        \end{tikzpicture}
        \begin{tikzpicture}[caption=$\underline{\sigma}$]
            \axes{4}{4}{4}
            \planepartition{{3,2},{3,2},{1}}{lightgray}
        \end{tikzpicture}
    \end{center}
    \begin{center}
        \begin{tikzpicture}[caption=$\sigma$]
            \axes{4}{4}{4}
            \planepartition{{0,0,1},{0,0,1},{0,1,1}}{lightgray}
            \planepartition[1]{{0,0,1},{},{1,1}}{lightgray}
            \planepartition[2]{{},{},{1}}{lightgray}
        \end{tikzpicture}
        \begin{tikzpicture}[caption=$\sigma^\circ$]
            \axes{4}{4}{4}
            \planepartition{{0,0,1},{0,0,1},{0,1,1}}{lightgray}
        \end{tikzpicture}
    \end{center}
    We see clearly that $\sigma \smallsetminus \sigma^\circ$ has two connected components. Let $\tau$ be the leftmost component. Let $\overline{p}, \underline{p}$ be the height functions of $\sigma$ and $\overline{p}', \underline{p}'$ the height functions of $\tau - \wedge\tau$. Then, noting that $\wedge\tau = (0, 2, 1)$, \begin{alignat*}{5}
        \overline{p}'(0, 0) &= 2 &&= 3 - 1 &&= \overline{p}(0, 2) - 1 &&= \overline{p}((0, 0) + \pi(\wedge\tau)) - 1 \\
        \overline{p}'(1, 0) &= 1 &&= 2 - 1 &&= \overline{p}(1, 2) - 1 &&= \overline{p}((1, 0) + \pi(\wedge\tau)) - 1 \end{alignat*} and \begin{alignat*}{5}
        \underline{p}'(0, 0) &= 0 &&\geq 1 - 1 &&= \underline{p}(0, 2) - 1 &&= \underline{p}((0, 0) + \pi(\wedge\tau)) - 1 \\
        \underline{p}'(1, 0) &= 0 &&\geq 0 - 1 &&= \underline{p}(1, 2) - 1 &&= \underline{p}((1, 0) + \pi(\wedge\tau)) - 1.
    \end{alignat*}
\end{exa}

\begin{defn}\label{DEF_SLICING_b}
    Given a floor plan $(\nu, b, c)$, let $$\nu^* = \bigcup_{j = 1}^\ell\Sigma_{\nu_j}.$$ Enumerate $\nu^*$ arbitrarily as $\nu^*_j$, $j \in [\ell^*]$ and let $\eta : [\ell^*] \to [\ell]$ be such that for each $j \in [\ell^*]$, $\nu^*_j \in \Sigma_{\nu_{\eta(j)}}$. Correspondingly, for each $j \in [\ell^*]$, let $\tau_j$ be the connected component of $\nu_{\eta(j)} \smallsetminus \nu_{\eta(j)}^\circ$ defining $\nu^*_j$, like in \Cref{DEF_SLICING_a}. In particular, $\tau_j - \wedge\tau_j = \nu^*_j$. Define $b^*_j$, and $c^*_j$ as $b^*_j = b_{\eta(j)} + \pi(\wedge\tau_j)$ and $c^*_j = c_{\eta(j)} + \pi(\wedge\tau_j)$. Call $(\nu^*, b^*, c^*)$ along with $\eta$ a bottom slice reduction of $(\nu, b, c)$. For convenience, we will suppress $\eta$ and simply call $(\nu^*, b^*, c^*)$ the bottom slice reduction of $(\nu, b, c)$.
\end{defn}

\begin{exa}
    Suppose $(\nu, b, c)$ were a floor plan with $\nu$ given as follows:
    
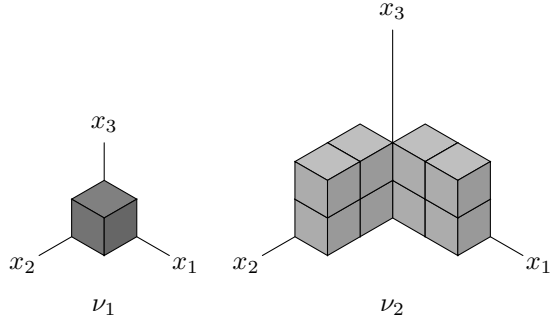
\begin{center}
        \begin{tikzpicture}[caption=$\nu_1$]
            \axes{2}{2}{2}
            \planepartition{{1}}{gray}
        \end{tikzpicture}
        \begin{tikzpicture}[caption=$\nu_2$]
            \axes{4}{4}{4}
            \planepartition{{0,2,2},{2,0,0},{2,0,0}}{lightgray}
        \end{tikzpicture}
    \end{center}
    Suppose too that $b$ were given by $b_1 = (2, 2)$ and $b_2 = (0, 0)$, so that altogether we have the following:
    
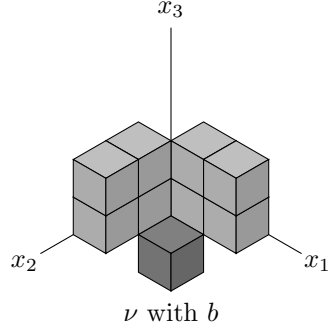
\begin{center}
        \begin{tikzpicture}[caption={$\nu$ with $b$}]
            \axes{4}{4}{4}
            \planepartition{{0,2,2},{2,0,0},{2,0,0}}{lightgray}
            \planepartition{{0,0,0},{0,0,0},{0,0,1}}{gray}
        \end{tikzpicture}
    \end{center}
    Then, observe that $\Sigma_{\nu_1} = \varnothing$, and that the connected components defining $\nu_2 \smallsetminus \nu_2^\circ$ are as follows, enumerated arbitrarily:
    
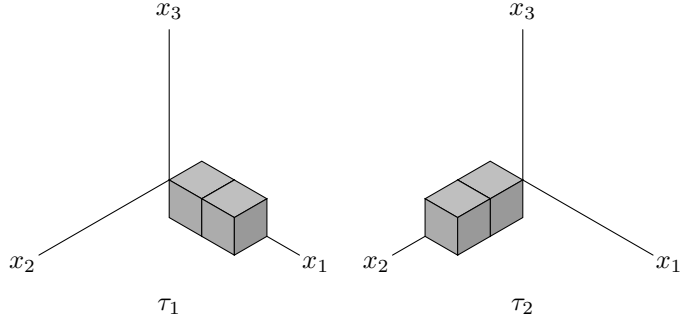
\begin{center}
        \begin{tikzpicture}[caption=$\tau_1$]
            \axes{4}{4}{4}
            \planepartition{{0,1,1}}{lightgray}
        \end{tikzpicture}
        \begin{tikzpicture}[caption=$\tau_2$]
            \axes{4}{4}{4}
            \planepartition{{0,0,0},{1,0,0},{1,0,0}}{lightgray}
        \end{tikzpicture}
    \end{center}
    Then, $$\nu^* = \Sigma_{\nu_2} = \{\nu_1^* = \tau_1 - \wedge\tau_1, \nu_2^* = \tau_2 - \wedge\tau_2\},$$ and $\eta$ is given by $\eta(1) = 2$ and $\eta(2) = 2$. Moreover, $\beta^*_1 = \beta_{\eta(1)} + \pi(\wedge\tau_1) = (0, 0) + (1, 0) = (1, 0)$, and similarly $\beta^*_2 = (0, 1)$, which altogether gives:
    
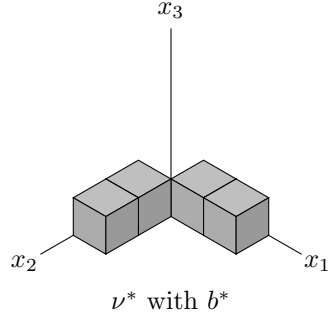
\begin{center}
        \begin{tikzpicture}[caption={$\nu^*$ with $b^*$}]
            \axes{4}{4}{4}
            \planepartition{{0,1,1},{1,0,0},{1,0,0}}{lightgray}
        \end{tikzpicture}
    \end{center}
\end{exa}

\begin{lem}\label{LEM_USEFULPROOFtau}
    Given a floor plan $(\nu, b, c)$, its bottom slice reduction $(\nu^*, b^*, c^*)$ is also a floor plan.
\end{lem}

\begin{proof}
    Indeed, for all $j \in [\ell^*]$, $\nu^*_j$ is a connected abstract skew shape. Note that $$\pi(\nu^*_j) + b^*_j = \pi\left(\tau_j - \wedge\tau_j\right) + b_{\eta(j)} + \pi\left(\wedge\tau_j\right) = \pi(\tau_j) + b_{\eta(j)} \subseteq \pi(\nu_{\eta(j)}) + b_{\eta(j)}.$$ Now, take $j' \in [\ell^*]$, $j' \neq j$. If $\eta(j) \neq \eta(j')$, then $\pi(\nu^*_j) + b^*_j$ and $\pi(\nu^*_{j'}) + b^*_{j'}$ are disjoint because $\pi(\nu_{\eta(j)}) + b_{\eta(j)}$ and $\pi(\nu_{\eta(j')}) + b_{\eta(j')}$ are by definition of a floor plan. Otherwise, $\pi(\nu^*_j) + b^*_j$ and $\pi(\nu^*_{j'}) + b^*_{j'}$ are still disjoint as $\tau_j$ and $\tau_{j'}$ are closed under $\geq$ in $\nu_{\eta(j)} + b_{\eta(j)}$. By symmetry, the same statements hold for $\pi(\nu_j^*) + c_j$, $j \in [\ell]$, and hence the bottom slice reduction is a floor plan.
\end{proof}

\begin{thm}\label{THM_RED}
    If $(\nu, b, c)$ is a floor plan and $(\nu^*, b^*, c^*)$ is its bottom slice reduction, then for all $a \in \bbN^2$, $\overline{p}_\lambda(a) \geq \overline{p}_{\lambda^*}(a) + 1$, and analogously for $\overline{p}_{\mu^*}$ where $\lambda^*$ and $\mu^*$ arise from the canonical realization of $(\nu^*, b^*, c^*)$ and $\lambda$ and $\mu$ arise from the canonical realization of $(\nu, b, c)$.
\end{thm}

\begin{proof}
    Denote by $\overline{p}_j, \underline{p}_j$ the upper and lower height functions of $\nu_j$ for $j \in [\ell]$. Likewise, denote by $\overline{p}_j^*, \underline{p}_j^*$ the upper and lower height functions of $\nu_j^*$ for $j \in [\ell^*]$. 

    We claim first that for $j, j' \in [\ell^*], j \neq j'$ with $\eta(j) = \eta(j')$, if $\nu^*_j <_{b^*} \nu^*_{j'}$ then $h_{b^*}(j, j') \leq 0$. Suppose otherwise. Then, by definition, there is some $v \in \pi(\nu^*_j) + b^*_j$ and $w \in \pi(\nu^*_{j'}) + b^*_{j'}$ with $v \leq w$ for which $\overline{p}_{j'}^*(w - b^*_{j'}) > \underline{p}_j^*(v - b^*_j)$. Then, by \Cref{REM_HTOFCOM}, $$\overline{p}_{j'}^*(w - b^*_{j'}) = \overline{p}_{\eta(j)}(w - b^*_{j'} + \pi(\wedge\tau_{j'})) - 1 = \overline{p}_{\eta(j)}(w - b_{\eta(j)}) - 1$$ and $$\underline{p}_j^*(v - b^*_j) \geq \underline{p}_{\eta(j)}(v - b^*_j + \pi(\wedge\tau_j)) - 1 = \underline{p}_{\eta(j)}(v - b_{\eta(j)}) - 1,$$ meaning $\overline{p}_{\eta(j)}(w - b_{\eta(j)}) > \underline{p}_{\eta(j)}(v - b_{\eta(j)})$. Now, let $$d = \max\{\underline{p}_{\eta(j)}(v - b_{\eta(j)}), 1\}.$$ Remark that for there to exist an element $v \in \pi(\nu^*_j) + b^*_j$, it must be that $\overline{p}_j^*(v - b^*_j) > 0$, which yields $\overline{p}_{\eta(j)}(v - b_{\eta(j)}) - 1 > 0$. By the nondecreasing nature of $\overline{p}_{\eta(j)}$, it follows that $$\overline{p}_{\eta(j)}(w - b_{\eta(j)}) \geq \overline{p}_{\eta(j)}(v - b_{\eta(j)}) > 1.$$ Then, as $\overline{p}_{\eta(j)}(v - b_{\eta(j)}) > 1$, we have that $\overline{p}_{\eta(j)}(v - b_{\eta(j)}) > \underline{p}_{\eta(j)}(v - b_{\eta(j)})$, as otherwise if they are equal then they are both zero by \Cref{LEM_PPROPS}. So, $d$ satisfies \begin{align*}
        \overline{p}_{\eta(j)}(w - b_{\eta(j)}) &\geq \overline{p}_{\eta(j)}(v - b_{\eta(j)}) \\
        &> \max\{\underline{p}_{\eta(j)}(v - b_{\eta(j)}), 1\} \\
        &= d \\
        &\geq \underline{p}_{\eta(j)}(v - b_{\eta(j)}).
    \end{align*} Hence, for all $v \leq u \leq w$, $$\overline{p}_{\eta(j)}(u - b_{\eta(j)}) \geq \overline{p}_{\eta(j)}(w - b_{\eta(j)}) > d \geq \underline{p}_{\eta(j)}(v - b_{\eta(j)}) \geq \underline{p}_{\eta(j)}(u - b_{\eta(j)}).$$ By \Cref{LEM_USEPS}, $(u - b_{\eta(j)}, d) \in \nu_{\eta(j)}$, but observe now that as $d \geq 1$ and $v \leq u \leq w$ is arbitrary, we have constructed a path from $v$ to $w$ in $\nu_{\eta_j} \smallsetminus \nu_{\eta_j}^\circ$. In other words, $\tau_j$ and $\tau_{j'}$ are connected in $\nu_{\eta(j)} \smallsetminus \nu^\circ_{\eta(j)}$, a contradiction.

    Next, we claim that for $j, j' \in [\ell^*], j \neq j'$, $h_{b^*}(j, j') \leq h_b({\eta(j)}, {\eta(j')})$. Indeed, this is clear, as every witness of the former is a witness of the latter. Note too that in the case of $j \neq j'$, $\nu^*_j <_{b^*} \nu^*_{j'}$ implies $\nu_{\eta(j)} <_b \nu_{\eta(j')}$ trivially.

    From the above two claims, it follows that for $j \in [\ell^*]$, $h_{b^*}(j) \leq h_b({\eta(j)})$. Indeed, let $i_0, \cdots, i_m \in [\ell^*]$, $\nu^*_{i_0} <_{b^*} \cdots <_{b^*} \nu^*_{i_m}$, $i_0 = j$ be a chain in $\leq_{b^*}$ of minimal length with $$h_{b^*}(j) =  \sum_{r = 1}^m h_{b^*}(i_{r - 1}, i_r).$$ Then, for all $r \in [m]$, $\eta(i_{r - 1}) \neq \eta(i_r)$, as otherwise by the first claim $h_{b^*}(i_{r - 1}, i_r) \leq 0$, and so omitting $i_r$ we contradict minimality. Then, by the second claim, $$\sum_{r = 1}^m h_{b^*}(i_{r - 1}, i_r) \leq \sum_{r = 1}^m h_b(\eta(i_{r - 1}), \eta(i_r)),$$ where $\eta(i_0), \cdots, \eta(i_m) \in [\ell]$, $\nu_{\eta(i_0)} <_b \cdots <_b \nu_{\eta(i_m)}$, $\eta(i_0) = \eta(j)$ is a chain in $\leq_b$ starting at $\eta(j)$, and we have the desired result.
    
    Hence, by \Cref{LEM_HTEXP}, \begin{align*}
        \overline{p}_{\lambda^*}(a) &= \max_{\substack{j \in [\ell^*] \\ v \in \pi(\nu^*_j) + b^*_j \\ a \leq v}} (h_{b^*}(j) + \overline{p}_j^*(v - b^*_j)) \\
        &\leq \max_{\substack{j \in [\ell^*] \\ v \in \pi(\nu^*_j) + b^*_j \\ a \leq v}} (h_b(\eta(j)) + \overline{p}_j^*(v - b^*_j)) \\
        &= \max_{\substack{j \in [\ell^*] \\ v \in \pi(\nu^*_j) + b^*_j \\ a \leq v}} (h_b(\eta(j)) + \overline{p}_{\eta(j)}(v - b^*_j + \pi(\wedge\tau_j)) - 1) \tag*{by \Cref{REM_HTOFCOM}} \\
        &= \max_{\substack{\eta(j) \in [\ell] \\ v \in \pi(\nu_{\eta(j)} \smallsetminus \nu^\circ_{\eta(j)}) + b_{\eta(j)} \\ a \leq v}} (h_b(\eta(j)) + \overline{p}_{\eta(j)}(v - b_{\eta(j)})) - 1 \\
        &\leq \max_{\substack{j \in [\ell] \\ v \in \pi(\nu_j) + b_j \\ a \leq v}} (h_b(j) + \overline{p}_j(v - b_j)) - 1 \\
        &= \overline{p}_\lambda(a) - 1.
    \end{align*} The proof is identical for $\overline{p}_{\mu^*}$.
\end{proof}

\begin{prop}\label{PROP_MAIN}
    Let $(\nu, b, c)$ be a counterexample and $(\nu^*, b^*, c^*)$ its bottom slice reduction. Denote by $|\nu^\circ|$ the sum $|\nu^\circ_1| + \cdots + |\nu^\circ_\ell|$. If $|\nu^\circ| \leq |\lambda^\circ \cap \mu^\circ|$, then $(\nu^*, b^*, c^*)$ is also a counterexample.
\end{prop}

\begin{proof}
    By \Cref{THM_RED}, as $\overline{p}_\lambda(a) \geq \overline{p}_{\lambda^*}(a) + 1$, we have $\overline{p}_\lambda(a) - 1 \geq \overline{p}_{\lambda^*}(a)$, which implies that $\lambda \smallsetminus \lambda^\circ \supseteq \lambda^*$. Similarly, $\mu \smallsetminus \mu^\circ \supseteq \mu^*$, and \begin{align*}
        |\lambda^* \cap \mu^*| &= \sum_{a \in \bbN^2}\min\{\overline{p}_{\lambda^*}(a), \overline{p}_{\mu^*}(a)\} \\
        &\leq \sum_{a \in \bbN^2}\min\{\overline{p}_{\lambda}(a) - 1, \overline{p}_{\mu}(a) - 1\} \\
        &= \sum_{a \in \bbN^2}\min\{\overline{p}_{\lambda}(a), \overline{p}_{\mu}(a)\} - |\{a \in \bbN^2 : \overline{p}_{\lambda}(a) > 0, \overline{p}_{\mu}(a) > 0\}| \\
        &= |\lambda \cap \mu| - |\lambda^\circ \cap \mu^\circ|.
    \end{align*}

    Now, observe that $|\nu^*| = |\nu| - |\nu^\circ|$, and as $\pi(\nu_1) + b_1, \cdots, \pi(\nu_\ell) + b_\ell$ are are disjoint and contained in $\lambda^\circ$, so too are $\pi(\nu^\circ_1) + b_1, \cdots, \pi(\nu^\circ_\ell) + b_\ell$. Similarly, $\pi(\nu^\circ_1) + c_1, \cdots, \pi(\nu^\circ_\ell) + c_\ell$ are disjoint and contained in $\mu^\circ$. In particular, $|\nu^\circ| \leq |\lambda^\circ|, |\mu^\circ|$.

    Then: \begin{align*}
        |\nu^*| &= |\nu| - |\nu^\circ| \\
        &\geq |\nu| - |\lambda^\circ \cap \mu^\circ| \text{ by assumption} \\
        &\geq |\nu| - (|\lambda \cap \mu| - |\lambda^* \cap \mu^*|) \\
        &= (|\nu| - |\lambda \cap \mu|) + |\lambda^* - \mu^*| \\
        &> |\lambda^* - \mu^*| \text{ as $(\nu, b, c)$ is a counterexample},
    \end{align*} meaning $(\nu^*, b^*, c^*)$ is a counterexample, as desired.
\end{proof}

\section{Two-dimensional resolution}

We will show in this section that the assumption made in \Cref{PROP_MAIN} holds for a large class of combinatorial modules.

\begin{defn}
    Let $\nu^\circ$ be a sequence $\nu_1^\circ, \cdots, \nu_\ell^\circ \subseteq \bbN^2$ of two-dimensional connected abstract skew shapes and let $b, c$ be sequences in $\bbN^2$ of length $\ell$. Define the partial order $\leq_b$ as per \Cref{DEF_main_partial_order}. We say that $(\nu^\circ, b, c)$ is a right-free configuration if the $\nu^\circ_j$ are incomparable under $\leq_b$ and the $\nu^\circ_j + c_j$ are disjoint for $j \in [\ell]$. We denote by $\lambda^\circ$ the closure of $(\nu_1^\circ + b_1) \sqcup \cdots \sqcup (\nu_\ell^\circ + b_\ell)$ under $\leq$, and similarly for $\nu^\circ$.
\end{defn}

\begin{exa}\label{section 5: reusable example}
    Let $\nu^\circ$ be the following sequence:
    
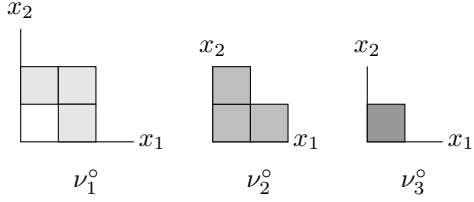
\begin{center}
        \begin{tikzpicture}[caption=$\nu_1^\circ$]
            \axestwo{3}{3}
            \emptybox{0}{1}{gray!20!white}
            \emptybox{1}{1}{gray!20!white}
            \emptybox{1}{0}{gray!20!white}
        \end{tikzpicture}
        \begin{tikzpicture}[caption=$\nu_2^\circ$]
            \axestwo{2}{2}
            \emptybox{0}{0}{gray!50!white}
            \emptybox{1}{0}{gray!50!white}
            \emptybox{0}{1}{gray!50!white}
        \end{tikzpicture}
        \begin{tikzpicture}[caption=$\nu_3^\circ$]
            \axestwo{2}{2}
            \emptybox{0}{0}{gray!80!white}
        \end{tikzpicture}
    \end{center}
    Let $b = ((0, 4), (2, 2), (4, 0))$ and $c = ((0, 3), (0, 1), (0, 0))$, yielding the following $\lambda^\circ$ and $\mu^\circ$, with $\nu^\circ$ coloured:
    
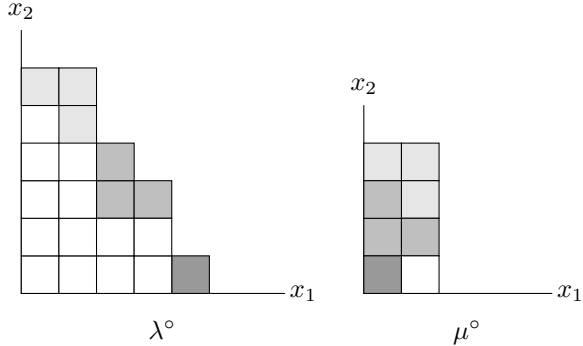
\begin{center}
        \begin{tikzpicture}[caption=$\lambda^\circ$]
            \axestwo{7}{7}
            \emptybox{0}{5}{gray!20!white}
            \emptybox{1}{5}{gray!20!white}
            \emptybox{1}{4}{gray!20!white}
            \emptybox{2}{2}{gray!50!white}
            \emptybox{3}{2}{gray!50!white}
            \emptybox{2}{3}{gray!50!white}
            \emptybox{4}{0}{gray!80!white}
            \emptybox{0}{0}{white}
            \emptybox{0}{1}{white}
            \emptybox{0}{2}{white}
            \emptybox{0}{3}{white}
            \emptybox{1}{0}{white}
            \emptybox{2}{0}{white}
            \emptybox{3}{0}{white}
            \emptybox{1}{1}{white}
            \emptybox{1}{2}{white}
            \emptybox{2}{1}{white}
            \emptybox{3}{1}{white}
        \end{tikzpicture}
        \begin{tikzpicture}[caption=$\mu^\circ$]
            \axestwo{5}{5}
            \emptybox{0}{3}{gray!20!white}
            \emptybox{1}{3}{gray!20!white}
            \emptybox{1}{2}{gray!20!white}
            \emptybox{0}{1}{gray!50!white}
            \emptybox{1}{1}{gray!50!white}
            \emptybox{0}{2}{gray!50!white}
            \emptybox{0}{0}{gray!80!white}
            \emptybox{1}{0}{white}
        \end{tikzpicture}
    \end{center}
    
    Observe that the $\nu_j^\circ + c_j$ for $1 \leq j \leq 3$ are pairwise disjoint, while the $\nu_j^\circ + b_j$ are not only pairwise disjoint, but no two elements belonging to different connected skew shapes are comparable. This justifies the choice of language in calling $(\nu, b, c)$ a right-free configuration.

    Also note that $|\lambda^\circ \cap \mu^\circ| = 8 \geq 7 = |\nu^\circ|$.
\end{exa}

Say that a right-free configuration $(\nu^\circ, b, c)$ has small intersection if $$|\lambda^\circ \cap \mu^\circ| < |\nu^\circ|.$$ We claim no such configuration exists.

\begin{defn}\label{Section 5: poset order}
    Define a partial order $\leq$ on right-free configurations, where $({\nu^\circ}', b', c') \leq (\nu^\circ, b, c)$ if the following conditions hold:
    \begin{enumerate}[(i)]
        \item ${\lambda^\circ}' \subseteq \lambda^\circ$ and ${\mu^\circ}' \subseteq \mu^\circ$\
        \item $|{\nu^\circ}'| \leq |\nu^\circ|$
        \item $|{\lambda^\circ}' \cap {\mu^\circ}'| - |{\nu^\circ}'| \leq |\lambda^\circ \cap \mu^\circ| - |\nu^\circ|$
    \end{enumerate}
\end{defn}

\begin{defn}
    For $\sigma \subseteq \bbN^n$ any subset, define the socle elements of $\sigma$ to be $$\Soc(\sigma) = \{v \in \sigma: v + e_1, \cdots, x + e_n \not\in \sigma\}.$$
    Assume now $n = 2$. For $i \in \bbN$, we define the $i$th row of $\sigma$ to be
    \begin{align*}
        \row_i(\sigma) &= \{(v_1, v_2) \in \sigma : v_2 = i\}.
    \end{align*}
    Also, define the height of $\sigma$ to be 
    \begin{align*}
        H(\sigma) = \max\{j \in \bbN: (0,j) \in \sigma\}.
    \end{align*}
    Let $\pi_1,\pi_2 :\bbN^2 \to \bbN$ be the projections onto the first and second components respectively.
\end{defn}

\begin{exa}
    Take the right-free configuration $(\nu^\circ, b, c)$ as given in \Cref{section 5: reusable example}. Note that $H(\lambda^\circ) = 5$ and $\row_1(\lambda^\circ) = \varnothing$. For $1 \leq j \leq 3$, we define $$b_j' = \begin{cases}b_j - e_2, \pi_2(b_j) > 1 \\ b_j \text{ otherwise,}\end{cases}$$ yielding $b' = ((0, 3), (2, 1), (4, 0))$. This gives:
        
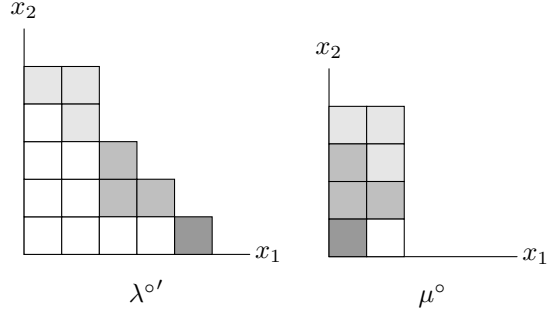
\begin{center}
        \begin{tikzpicture}[caption=${\lambda^\circ}'$]
            \axestwo{6}{6}
            \emptybox{0}{4}{gray!20!white}
            \emptybox{1}{4}{gray!20!white}
            \emptybox{1}{3}{gray!20!white}
            \emptybox{2}{1}{gray!50!white}
            \emptybox{3}{1}{gray!50!white}
            \emptybox{2}{2}{gray!50!white}
            \emptybox{4}{0}{gray!80!white}
            \emptybox{0}{0}{white}
            \emptybox{0}{1}{white}
            \emptybox{0}{2}{white}
            \emptybox{0}{3}{white}
            \emptybox{1}{0}{white}
            \emptybox{2}{0}{white}
            \emptybox{3}{0}{white}
            \emptybox{1}{1}{white}
            \emptybox{1}{2}{white}
        \end{tikzpicture}
        \begin{tikzpicture}[caption=$\mu^\circ$]
            \axestwo{5}{5}
            \emptybox{0}{3}{gray!20!white}
            \emptybox{1}{3}{gray!20!white}
            \emptybox{1}{2}{gray!20!white}
            \emptybox{0}{1}{gray!50!white}
            \emptybox{1}{1}{gray!50!white}
            \emptybox{0}{2}{gray!50!white}
            \emptybox{0}{0}{gray!80!white}
            \emptybox{1}{0}{white}
        \end{tikzpicture}
    \end{center}
    Observe that $(\nu^\circ, b', c)$ is still a right-free configuration, that $\row_1({\lambda^\circ}') \neq \varnothing$, and $(\nu^\circ, b', c) < (\nu^\circ, b, c)$.
\end{exa}

\begin{lem}\label{LEM_tedious_but_useful}
    Suppose $(\nu^\circ, b, c)$ is a right-free configuration with small intersection, and suppose that it is a minimal such configuration under the order defined in \Cref{Section 5: poset order}. Then the following hold:
    \begin{enumerate}[(a)]
        \item $\Soc(\lambda^\circ) \cap \mu^\circ = \varnothing$ and $\Soc(\mu^\circ) \cap \lambda^\circ = \varnothing$
        \item For $r, s \in \bbN$ with $r \leq H(\lambda^\circ)$ and $s \leq H(\mu^\circ)$, we have 
        \[
        \row_r(X), \row_s(Y)\neq \varnothing,
        \]
        where $$X = \bigsqcup_{j = 1}^\ell (\nu_j^\circ + b_j) \text{ and } Y = \bigsqcup_{j = 1}^\ell (\nu_j^\circ + c_j).$$
    \end{enumerate}
\end{lem}
\begin{proof}
Suppose for the sake of contradiction that $\Soc(\lambda^\circ) \cap \mu^\circ \neq \varnothing$. Observe that if $v$ is in the intersection, then $v \in \Soc(\lambda^\circ)$ implies $v \in \Soc(\nu^\circ_j + b_j)$ for some $j \in \ell$ by construction of $\lambda^\circ$. Then, forming ${\nu^\circ}'$ by removing $v - b_j$ from $\nu^\circ_j$, observe that $({\nu^\circ}', b, c)$ is a right-free configuration satisfying \Cref{Section 5: poset order} \textit{(i)} and \textit{(ii)}. \Cref{Section 5: poset order} \textit{(iii)} follows immediately from the observations that $|{\lambda^\circ}' \cap {\mu^\circ}'| = |\lambda^\circ \cap \mu^\circ| - 1$ and $|{\nu^\circ}'| = |\nu^\circ| - 1$.

To show \textit{(b)}, we will prove that $\row_r(X) \neq \varnothing$ for all $r \leq H(\lambda^\circ)$. The other statement follows by a weaker argument. Suppose towards a contradiction that there exists $r \leq H(\lambda^\circ)$ for which $\row_r(X) = \varnothing$. Define $$b_j' = \begin{cases}b_j - e_2, \pi_2(b_j) > r \\ b_j \text{ otherwise.}\end{cases}$$

We claim that $(\nu^\circ, b', c)$ is a right-free configuration. Indeed, as $(\nu^\circ, b, c)$ is a right-free configuration, the only way for $(\nu^\circ, b', c)$ to not be a right-free configuration is if there were some $i, j \in [\ell]$, $\pi_2(b_i) \leq r$, $\pi_2(b_j) > r$ for which $\nu_i^\circ$ and $\nu_j^\circ$ are comparable under $\leq_{b'}$. Suppose this were the case. Take $u \in \nu_i^\circ + b_i'$ and $v \in \nu_j^\circ + b_j'$ comparable under $\leq$. By definition, $b_i' = b_i$ and $b_j' = b_j - e_2$. Clearly, as $\row_r(X) = \varnothing$, $\pi_2(b_i) < r$. Then, as $\nu_i^\circ$ is a connected abstract skew shape, it follows that $\pi_2(u) < r$ too. Similarly, $\pi_2(v) \geq \pi_2(b_j') = \pi_2(b_j) - 1 \geq r$. Hence, for $u$ and $v$ to be comparable under $\leq$, it must be that $u \leq v$, so $\pi_1(u) \leq \pi_1(v)$. But now observe that $\pi_2(u) < \pi_2(v) < \pi_2(v + e_2)$, so $u \leq v + e_2$, but $u \in \nu_i^\circ + b_i$ and $v + e_2 \in \nu_j^\circ + b_j$, a contradiction.

Let ${\lambda^\circ}'$ be the closure under $\leq$ of $\nu^\circ_1 + b'_1 \sqcup \cdots \sqcup \nu^\circ_\ell + b'_\ell$. As $\lambda^\circ$ is the closure under $\leq$ of $\nu_1^\circ + b_1, \cdots, \nu_\ell^\circ + b_\ell$, there exists $j \in [\ell]$ for which with $\pi_2(b_j) > r$, which implies that ${\lambda^\circ}' \subset \lambda^\circ$. Then, ${\lambda^\circ}' \cap \mu^\circ \subseteq  \lambda^\circ \cap \mu^\circ$, meaning $$|{\lambda^\circ}' \cap \mu^\circ| - |{\nu^\circ}| \leq |\lambda^\circ \cap \mu^\circ| - |\nu^\circ|,$$ and so $(\nu^\circ, b', c) < (\nu^\circ, b, c)$ is of small intersection, contradicting minimality.
\end{proof}

\begin{lem}
   Suppose $(\nu^\circ, b, c)$ is a right-free configuration with small intersection, and suppose that it is a minimal such configuration under the order defined in \Cref{Section 5: poset order}. Then, 
    $$H(\lambda^\circ) = \sum_{j=1}^\ell H(\nu_j^\circ).$$
    Moreover, $H(\mu^\circ) \leq  H(\lambda^\circ)$.
\end{lem}
\begin{proof}
First, observe that
$$\pi_2\left(\bigcup_{j=1}^\ell (\nu_j^\circ + b_j)\right)\subseteq \pi_2(\lambda^\circ) \text{ and } \pi_2\left(\bigcup_{j=1}^\ell (\nu_j^\circ + c_j)\right)\subseteq \pi_2(\mu^\circ).$$
We claim these are equalities. Indeed, suppose the first inequality were strict. Take $$r \in \pi_2(\lambda^\circ) \smallsetminus \pi_2\left(\bigcup_{j=1}^\ell (\nu_j^\circ + b_j)\right).$$ Then, $r \leq H(\lambda^\circ)$ but $$\row_r\left(\bigcup_{j=1}^\ell (\nu_j^\circ + b_j)\right) = \varnothing,$$ contradicting \Cref{LEM_tedious_but_useful} \textit{(b)}. The second equality follows identically. Now:
\begin{align*}
    \sum_{j=1}^\ell H(\nu_i^\circ) &= \sum_{j=1}^\ell |\pi_2(\nu_j^\circ)| \text{ as $\nu_1^\circ, \cdots, \nu_\ell^\circ$ are connected abstract skew shapes}\\
    &=\sum_{j=1}^\ell |\pi_2(\nu_j^\circ + b_j)|\\
    &=\left|\pi_2\left(\bigcup_{j=1}^\ell \nu_j^\circ + b_j\right)\right| \text{ as $\nu_1^\circ + b_1, \cdots, \nu_\ell^\circ + b_\ell$ are incomparable under $\leq_b$}\\
    &= |\pi_2(\lambda^\circ)| \text{ by the observation above} \\ &= H(\lambda^\circ) \text{ as $\lambda^\circ$ is connected}.
\end{align*}

Then, arguing similarly,
$$H(\mu^\circ)= |\pi_2(\mu^\circ)| = \left|\pi_2\left(\bigcup_{j=1}^\ell (\nu_j^\circ + c_j)\right)\right| \leq \sum_{j=1}^\ell |\pi_2(\nu_j^\circ + c_j)| = \sum_{j=1}^\ell |\pi_2(\nu_j^\circ| = H(\lambda^\circ),$$ as desired.
\end{proof}

\begin{lem}\label{Section 5: intersection of top row of mu}
     Suppose $(\nu^\circ, b, c)$ is a right-free configuration with small intersection, and suppose that it is a minimal such configuration under the order defined in \Cref{Section 5: poset order}. Then, $$\row_{H(\mu^\circ)}(\lambda^\circ) \subseteq \row_{H(\mu^\circ)}(\mu^\circ).$$ In particular, $$|\row_{H(\mu^\circ)}(\lambda^\circ \cap \mu^\circ)| = |\row_{H(\mu^\circ)}(\lambda^\circ)|.$$
\end{lem}
\begin{proof}
    Note that at least one of $\row_{H(\mu^\circ)}(\lambda^\circ) \subseteq \row_{H(\mu^\circ)}(\mu^\circ)$ and $\row_{H(\mu^\circ)}(\mu^\circ) \subseteq \row_{H(\mu^\circ)}(\lambda^\circ)$ must be true. Take $p \in \row_{H(\mu^\circ)}(\mu^\circ)$ a maximal element. Then, $p \in \Soc(\mu^\circ)$, so by \Cref{LEM_tedious_but_useful} \textit{(a)}, $p \not\in \lambda^\circ$. Thus, $\row_{H(\mu^\circ)}(\mu^\circ) \not\subseteq \row_{H(\mu^\circ)}(\lambda^\circ)$, and hence $\row_{H(\mu^\circ)}(\lambda^\circ) \subseteq \row_{H(\mu^\circ)}(\mu^\circ)$.
\end{proof}

\begin{thm}\label{Section 5: main}
    There are no right-free configurations of small intersection.
\end{thm}
\begin{proof}
    Suppose otherwise. Take $(\nu^\circ, b, c)$ is a minimal right-free configuration with small intersection under the order defined in \Cref{Section 5: poset order}. For each row index $r \in [H(\mu^\circ)]$, let $\Theta_r \subseteq [\ell]$ be the set of indices $j$ for which $\nu_j + c_j$ is present in the $r$th row of $\mu^\circ$; i.e., for all $j \in \Theta_r$, $\row_r(\nu_j^\circ + c_j) \neq \varnothing$, and $$\row_r\left(\bigcup_{j=1}^\ell (\nu_j^\circ + c_j)\right) = \bigcup_{j \in \Theta_r} \row_r(\nu_j^\circ + c_j).$$ Note too that $\row_r(\nu_j^\circ + c_j) = \row_{r - \pi_2(c_j)}(\nu_j^\circ) + c_j$.
    
    We claim that for all rows $r \in [H(\mu^\circ)]$, there exists a $j \in \Theta_r$ such that $r - \pi_2(c_j) + \pi_2(b_j) < H(\mu^\circ)$. 

    We show first that the claim suffices. Suppose the claim holds. Then, there exists a function $f:[H(\mu^\circ)] \to [H(\mu^\circ) -1]$ given by $f(r) = r - \pi_2(c_j) + \pi_2(b_j)$ for choices of $j \in \Theta_r$. We will show that $f$ is an injection, a direct contradiction. Suppose $f(r) = f(s)$; that is, for some $i \in \Theta_r$ and $j \in \Theta_s$, we have $r - \pi_2(c_i) + \pi_2(b_i) = s - \pi_2(c_j) + \pi_2(b_j)$. Call this common value $t$. Then, as $i \in \Theta_r$, $\row_r(\nu^\circ_i + c_i)$ is nonempty, so $\row_{r - \pi_2(c_i)}(\nu^\circ_i)$ is nonempty, so $\row_t(\nu^\circ_i + b_i)$ is nonempty. Similarly, $\row_t(\nu^\circ_j + b_j)$ is nonempty, so by incomparability under $\leq_b$, it follows that $i = j$, and hence $r = s$. Thus, $f$ is injective, and we are done.
    
    We turn to the proof of the claim. Suppose it doesn't hold and take $r \in [H(\mu^\circ)]$ such that for all $j \in \Theta_r$, $r - \pi_2(c_j) + \pi_2(b_j) \geq H(\mu^\circ)$. Observe that $$H(\mu^\circ) \leq r - \pi_2(c_j) + \pi_2(b_j) \leq H(\lambda^\circ),$$ as the $(r - \pi_2(c_j))$th row of $\nu_j^\circ$ is nonempty by construction of $\Theta_r$, and hence the $(r - \pi_2(c_j) + \pi_2(b_j))$th row of $\nu_j^\circ + b_j$ in $\lambda^\circ$ is nonempty, and $\lambda^\circ$ contains at least the closure of $\nu_j^\circ + b_j$ under $\leq$.

    Then, as $\lambda^\circ$ is a standard shape, we have that $$|\row_{r - \pi_2(c_j)}(\nu_j^\circ) + b_j| = |\row_{r - \pi_2(c_j) + \pi_2(b_j)}(\nu_j^\circ + b_j)| \leq |\row_{H(\mu^\circ)}(\lambda^\circ)|.$$ In fact, because $\nu_1^\circ, \cdots, \nu_\ell^\circ$ are incomparable under $\leq_b$, the sum over all $j \in \Theta_r$ of the $|\row_{r - \pi_2(c_j) + \pi_2(b_j)}(\nu_j^\circ + b_j)|$ must still be less than or equal to $|\row_{H(\mu^\circ)}(\lambda^\circ)|$, meaning
    \begin{align*}
         \sum_{j \in \Theta_r}|\row_{r - \pi_2(c_j)}(\nu_j^\circ)|&= \sum_{j \in \Theta_r}|\row_{r - \pi_2(c_j)}(\nu_j^\circ) + b_j|\\
         &\leq |\row_{H(\mu^\circ)}(\lambda^\circ)|\\
         &= |\row_{H(\mu^\circ)}(\lambda^\circ \cap \mu^\circ)| \text{ by \Cref{Section 5: intersection of top row of mu}.}
    \end{align*}

    To achieve a contradiction, we will construct another right-free configuration $({\nu^\circ}',b',c')$ of small intersection that is strictly smaller than $(\nu^\circ,b,c)$. In order to do so, we remove the $(r - \pi_2(c_j))$th row from each connected abstract skew shape $\nu_j^\circ$ for $j \in \Theta_r$. Let us formalize this process.
    
    For $j \in \Theta_r$, recall that we may write $\nu_j^\circ = \overline{\nu}_j^\circ \smallsetminus \underline{\nu}_j^\circ$ for standard shapes $\overline{\nu}_j^\circ, \underline{\nu}_j^\circ$. Observe that a standard shape $\sigma \subseteq \bbN^2$ can be represented by a nonincreasing sequence of natural numbers $(\alpha_0, \cdots, \alpha_m)$, where $\alpha_0 = |\row_0(\sigma)|, \cdots, \alpha_m = |\row_m(\sigma)|$. Say $\overline{\nu}_j^\circ$ has row sequence $(\alpha_0,\cdots,\alpha_{m_j})$ and $\underline{\nu}_j^\circ$ has row sequence $(\beta_0,\cdots,\beta_{m_j})$. Define the standard shapes $\overline{\rho}_j$ and $\underline{\rho}_j$ by the row sequences
    \begin{align*}
        (\alpha_0,\cdots,\alpha_{r-\pi_2(c_j)-1},\alpha_{r-\pi_2(c_j)+1},\cdots,\alpha_{m_j})& \text{ and}\\
        (\beta_0,\cdots,\beta_{r-\pi_2(c_j)-1},\beta_{r-\pi_2(c_j)+1},\cdots,\beta_{m_j})& \text{ respectively}.
    \end{align*}
    From this, define the skew shapes $$\rho_j = \begin{cases}\overline{\rho}_j \smallsetminus \underline{\rho}_j, j \in \Theta_r \\ \nu_j^\circ \text{ otherwise.}\end{cases}$$
    
    Observe that $\rho_j$ is precisely $\nu_j^\circ$ with the $(r - \pi_2(c_j))$th row removed, as desired, and that in particular, $\rho_j \subseteq \nu_j^\circ$. Let us now construct $({\nu^\circ}',b',c')$. Borrowing familiar notation from \Cref{DEF_SLICING_a} and \Cref{DEF_SLICING_b}, for $j \in [\ell]$, let $$\Sigma_j = \{\tau - \wedge\tau : \tau \text{ is a connected component of } \rho_j\}.$$ Note that for $j \not\in \Theta_r$, $\Sigma_j = \{\nu_j^\circ\}$. Let ${\nu^\circ}'$ be a sequence given an arbitrary enumeration of the set $\bigcup_{j = 1}^\ell\Sigma_j$. With $\ell' = |{\nu^\circ}'|$, let $\eta : [\ell'] \to [\ell]$ be such that for each $j \in [\ell']$, ${\nu^\circ_j}' \in \Sigma_{\eta(j)}$. Also, for each $j \in [\ell']$, let $\tau_j$ be the connected component of $\rho_{\eta(j)}$ for which $\tau_j - \wedge\tau_j = {\nu^\circ_j}'$. Note too that for $\eta(j) \not\in \Theta_r$, ${\nu^\circ_j}' = \tau_j = \nu^\circ_j$ and $\wedge\tau_j = 0$.

    Before proceeding, we define another set of indices: let $\Gamma_r \subseteq [\ell]$ be such that for all $j \in \Gamma_r$, $\pi_2(c_j) > r$. Note that $\Gamma_r$ and $\Theta_r$ are disjoint: indeed, it cannot be that both $\pi_2(c_j) > r$ and $\row_r(\nu_j^\circ + c_j) \neq \varnothing$. Now, we may finally define for $j \in [\ell']$:
    $$b_j' = b_{\eta(j)} + \wedge\tau_j, \text{ and}$$ $$c_j' = \begin{cases}
        c_{\eta(j)} + \wedge\tau_j, \eta(j) \not\in \Gamma_r \\
        c_{\eta(j)} + \wedge\tau_j - e_2 = c_{\eta(j)} - e_2 \text{ otherwise.}
    \end{cases}$$
    
    We show that $({\nu^\circ}',b',c')$ is a right-free configuration. It is not hard to see that as the $\nu^\circ_j$ are incomparable under $\leq_b$ for $j \in [\ell]$, so too are the ${\nu^\circ_j}'$ under $\leq_{b'}$ for $j \in [\ell']$. It remains to verify that ${\nu^\circ_i}' + c_i'$ and ${\nu^\circ_j}' + c_j'$ are disjoint for $i, j \in [\ell']$, $i \neq j$. Again, it is clear that the only nontrivial case is when $\eta(i) \in \Gamma_r$, $\eta(j) \not\in \Gamma_r$. Suppose this is the case and suppose for the sake of contradiction that they intersect. Unfolding definitions, we obtain $$(\nu^\circ_{\eta(i)} + c_{\eta(i)} - e_2) \cap ({\nu^\circ_j}' + c_j') \neq \varnothing.$$

    If $\eta(j) \not\in \Theta_r$, then the above becomes $$(\nu^\circ_{\eta(i)} + c_{\eta(i)} - e_2) \cap (\nu^\circ_{\eta(j)} + c_{\eta(j)}) \neq \varnothing$$ But, $\row_s(\nu^\circ_{\eta(i)} + c_{\eta(i)} - e_2) = \varnothing$ for $s < \pi_2(c_{\eta(i)} - e_2) = \pi_2(c_{\eta(i)}) - 1$, which is to say in particular for $s < r$, as $\pi_2(c_{\eta(i)}) - 1 \geq r$ due to $\eta(i)$ being in $\Gamma_r$. On the other hand, $\nu^\circ_{\eta(j)}$ being connected, $\pi_2(c_{\eta(j)}) \leq r$, and $\row_r(\nu^\circ_{\eta(j)} + c_{\eta(j)}) = \varnothing$ altogether imply that $\row_s(\nu^\circ_{\eta(j)} + c_{\eta(j)}) = \varnothing$ for $s \geq r$, contradicting the intersection being nonempty.

    Thus, $\eta(j) \in \Theta_r$. Take $p\in\nu^\circ_{\eta(i)}$ and $q \in {\nu^\circ_j}'$ such that $p + c_{\eta(i)} -e_2 = q + c_j'$. Since $\pi_2(c_{\eta(i)}) > r$, we have $$\pi_2(q + c_j')=\pi_2(p + c_{\eta(i)} -e_2) \geq \pi_2(c_{\eta(i)}) - 1 \geq r,$$ implying $\pi_2(q + \wedge\tau_j) \geq r - \pi_2(c_{\eta(j)})$. But recall that by definition, ${\nu^\circ_j}' + \wedge\tau_j$ is a connected component of $\rho_{\eta(j)}$. Moreover, through the construction of $\rho_{\eta(j)}$, we know that an element $q + \wedge\tau_j$ of $\rho_{\eta(j)}$ with $\pi_2(q + \wedge\tau_j) \geq r - \pi_2(c_{\eta(j)})$ must have come from the row above in $\nu^\circ_{\eta(j)}$, which is to say $q + \wedge\tau_j + e_2 \in \nu^\circ_{\eta(j)}$.

    At last, observe that $$\nu^\circ_{\eta(i)} + c_{\eta(i)} \ni p + c_{\eta(i)} = q + c_j' + e_2 = q + \wedge\tau_j + e_2 + c_{\eta(j)} \in \nu^\circ_{\eta(j)} + c_{\eta(j)},$$ contradicting disjointedness. Hence, $({\nu^\circ}',b',c')$ is a right-free configuration.


    By construction, it is easy to see that ${\mu^\circ}'$ is contained in $\mu^\circ$ minus its top row; i.e., ${\mu^\circ}' \subseteq \mu^\circ \smallsetminus \row_{H(\mu^\circ)}(\mu^\circ)$, and so ${\mu^\circ}' \subset \mu^\circ$. Also evident is that ${\lambda^\circ}' \subseteq \lambda^\circ$. Additionally, $$|{\nu^\circ}'| = |\nu^\circ| - \sum_{j \in \Theta_r} |\row_{r - \pi_2(c_j)}(\nu^\circ_j)|.$$ Hence: \begin{align*}
        |{\lambda^\circ}' \cap {\mu^\circ}'| &\leq |\lambda^\circ \cap (\mu^\circ \smallsetminus \row_{H(\mu^\circ)}(\mu^\circ))|\\
        &= |\lambda^\circ \cap \mu^\circ| - |\row_{H(\mu^\circ)}(\lambda^\circ \cap \mu^\circ)| \\
        &\leq |\lambda^\circ \cap \mu^\circ| - \sum_{j \in \Theta_r}|\row_{r - \pi_2(c_j)}(\nu_j^\circ)| \\
        \implies |{\lambda^\circ}' \cap {\mu^\circ}'| - |{\nu^\circ}'| &< |\lambda^\circ \cap \mu^\circ| - |\nu^\circ| < 0.
    \end{align*}
    
    That is, $({\nu^\circ}',b',c') < (\nu^\circ,b,c)$ is a strictly smaller right-free configuration of small intersection, contradicting minimality. So, we have our initial claim.
\end{proof}

\begin{defn}
    Say a floor plan $(\nu, b, c)$ is right-free if $h_b(j) \leq 0$ for $j \in [\ell]$.
\end{defn}

\begin{rem}\label{LEM_rf}
    If $(\nu, b, c)$ is a right-free floor plan, then $(\nu^\circ, b, c)$ is a right-free configuration, where $\nu^\circ$ is the sequence $\nu^\circ_1, \cdots, \nu^\circ_\ell$.
\end{rem}

\begin{thm}
    A right-free floor plan is not a counterexmample.
\end{thm}

\begin{proof}
    Suppose otherwise. If $|\nu| = 0$ then we are done. Let $(\nu, b, c)$ be a right-free floor plan that is a counterexample. Then, by \Cref{LEM_rf}, $(\nu^\circ, b, c)$ is a right-free configuration, so by \Cref{Section 5: main}, $|\nu^\circ| \leq |\lambda^\circ \cap \mu^\circ|$, and so by \Cref{PROP_MAIN}, the bottom slice reduction $(\nu^*, b^*, c^*)$ is also a counterexample. However, by the proof of \Cref{THM_RED}, $(\nu^*, b^*, c^*)$ is also a right-free floor plan, and by construction $|\nu^*| < |\nu|$, so repeating this process finitely many times, we will arrive at a contradiction, and so we are done.
\end{proof}

We are now equipped to work towards proving \Cref{BIGMAIN}:

\begin{lem}\label{LEM_indecdecomp}
    Let $I \subseteq K \subseteq S$ be monomial ideals of finite colength and let $\zeta$ be the corresponding skew shape to $K / I$. Let $\zeta_1, \cdots, \zeta_\ell$ be the connected components of $\zeta$, enumerated arbitrarily. Then, the $\bbN^3$-graded indecomposable decomposition of $K / I$ as an $S$-module is given by $P_1 \oplus \cdots \oplus P_\ell$, where for $j \in [\ell]$, $P_j$ is given by the $k$-span in $S / I$ of the monomials in $\zeta_j$.
\end{lem}

\begin{proof}
    It is clear that $K / I \cong P_1 \oplus \cdots \oplus P_\ell$ as $\bbN^3$-graded $S$-modules. 
    Hence, it suffices to prove that each $P_j$ is indecomposable an $\bbN^3$-graded $S$-module.

    Suppose $P_j = M \oplus M'$ is a decomposition. Let $a\in\bbN^3$ be such that $(P_j)_a\neq0$. Then, $\dim_k(P_j)_a=1$, and so either $M_a=(P_j)_a$ or $M'_a=(P_j)_a$. Without loss of generality, suppose $M_a=(P_j)_a$. We claim that $M=P_j$. To see this, let $v\in\zeta_j$ be the monomial corresponding to $a$, and let $w \in \zeta_j$ be any other monomial. By connectedness, we can take a path $v = v_0, \cdots, v_m = w$ in $\zeta_j$ from $v$ to $w$. Assume inductively that $v_{i - 1} \in M$ for some $i \in [m]$. Then, either $v_i = x_kv_{i - 1}$ in which case $v_i \in M$, or $x_kv_i = v_{i - 1}$ in which case if $v_i \in M'$, then $v_{i - 1} \in M'$, a contradiction; thus, $v_i \in M$.
    %
    %
%
\end{proof}

\begin{proof}[Proof of \Cref{BIGMAIN}]
    Observe that without loss of generality we may assume $N$ is scaffolded. Let $(\nu, b, c)$ be the floor plan for which $N$ is a realization. By \Cref{LEM_indecdecomp}, we may write $K / I = P_1 \oplus \cdots \oplus P_\ell$, where for each $j \in [\ell]$, $P_j$ is the span in $S / I$ of the monomials in $\zeta_j = \nu_j + \bm{b}_j$, where $\bm{b}_j = (b_j, h_b(j))$ by definition of the canonical realization. It is clear that $P_j \subseteq (x_3)$ if and only if $h_b(j) > 0$, so by assumption $(\nu, b, c)$ is a right-free floor plan, and we are done.
\end{proof}

\bibliographystyle{alpha}
\bibliography{refs}

@article{RS18,
title = {New classes of examples satisfying the three matrix analog of {G}erstenhaber's theorem},
journal = {Journal of Algebra},
volume = {516},
pages = {245-270},
year = {2018},
issn = {0021-8693},
doi = {https://doi.org/10.1016/j.jalgebra.2018.09.020},
url = {https://www.sciencedirect.com/science/article/pii/S0021869318305386},
author = {Jenna Rajchgot and Matthew Satriano}
}

@article{Ger61,
 ISSN = {0003486X, 19398980},
 URL = {http://www.jstor.org/stable/1970336},
 author = {Murray Gerstenhaber},
 journal = {Annals of Mathematics},
 number = {2},
 pages = {324--348},
 publisher = {[Annals of Mathematics, Trustees of Princeton University on Behalf of the Annals of Mathematics, Mathematics Department, Princeton University]},
 title = {On Dominance and Varieties of Commuting Matrices},
 urldate = {2025-02-27},
 volume = {73},
 year = {1961}
}

@Inbook{RSS20,
author="Rajchgot, Jenna
and Satriano, Matthew
and Shen, Wanchun",
editor="Acu, Bahar
and Danielli, Donatella
and Lewicka, Marta
and Pati, Arati
and Saraswathy RV
and Teboh-Ewungkem, Miranda",
title="Some Combinatorial Cases of the Three Matrix Analog of Gerstenhaber's Theorem",
bookTitle="Advances in Mathematical Sciences: AWM Research Symposium, Houston, TX, April 2019",
year="2020",
publisher="Springer International Publishing",
address="Cham",
pages="181--201",
isbn="978-3-030-42687-3",
doi="10.1007/978-3-030-42687-3_12",
url="https://doi.org/10.1007/978-3-030-42687-3_12"
}

@article{CSS24,
  title = {On the Algebra Generated by Three Commuting Matrices: Combinatorial Cases},
  volume = {31},
  ISSN = {1077-8926},
  url = {http://dx.doi.org/10.37236/12909},
  DOI = {10.37236/12909},
  number = {4},
  journal = {The Electronic Journal of Combinatorics},
  publisher = {The Electronic Journal of Combinatorics},
  author = {Cherny,  Ron and Satriano,  Matthew and Song,  Yohan},
  year = {2024},
  month = nov 
}

@article {sivicII,
    AUTHOR = {\v{S}ivic, Klemen},
     TITLE = {On varieties of commuting triples {II}},
   JOURNAL = {Linear Algebra Appl.},
  FJOURNAL = {Linear Algebra and its Applications},
    VOLUME = {437},
      YEAR = {2012},
    NUMBER = {2},
     PAGES = {461--489},
      ISSN = {0024-3795},
   MRCLASS = {15A27 (14A10 15A30)},
  MRNUMBER = {2921711},
MRREVIEWER = {Marcin Skrzy\'nski},
       DOI = {10.1016/j.laa.2011.08.014},
       URL = {https://doi.org/10.1016/j.laa.2011.08.014},
}

@article {Wadsworth,
    AUTHOR = {Wadsworth, Adrian R.},
     TITLE = {The algebra generated by two commuting matrices},
   JOURNAL = {Linear and Multilinear Algebra},
  FJOURNAL = {Linear and Multilinear Algebra},
    VOLUME = {27},
      YEAR = {1990},
    NUMBER = {3},
     PAGES = {159--162},
      ISSN = {0308-1087},
   MRCLASS = {16S50 (15A30 16R20)},
  MRNUMBER = {1064892},
MRREVIEWER = {William H. Gustafson},
       URL = {https://doi.org/10.1080/03081089008818007},
}

@article {BH,
    AUTHOR = {Barr\'{i}a, Jos\'{e} and Halmos, P. R.},
     TITLE = {Vector bases for two commuting matrices},
   JOURNAL = {Linear and Multilinear Algebra},
  FJOURNAL = {Linear and Multilinear Algebra},
    VOLUME = {27},
      YEAR = {1990},
    NUMBER = {3},
     PAGES = {147--157},
      ISSN = {0308-1087},
   MRCLASS = {16S50 (15A30 16R20)},
  MRNUMBER = {1064891},
MRREVIEWER = {William H. Gustafson},
       URL = {https://doi.org/10.1080/03081089008818006},
}

@article {GS,
    AUTHOR = {Guralnick, Robert M. and Sethuraman, B. A.},
     TITLE = {Commuting pairs and triples of matrices and related varieties},
   JOURNAL = {Linear Algebra Appl.},
  FJOURNAL = {Linear Algebra and its Applications},
    VOLUME = {310},
      YEAR = {2000},
    NUMBER = {1-3},
     PAGES = {139--148},
      ISSN = {0024-3795},
   MRCLASS = {15A30 (14M99 16S50)},
  MRNUMBER = {1753173},
MRREVIEWER = {A. R. Wadsworth},
       URL = {https://doi.org/10.1016/S0024-3795(00)00065-3},
}

@article {LL,
    AUTHOR = {Laffey, Thomas J. and Lazarus, Susan},
     TITLE = {Two-generated commutative matrix subalgebras},
   JOURNAL = {Linear Algebra Appl.},
  FJOURNAL = {Linear Algebra and its Applications},
    VOLUME = {147},
      YEAR = {1991},
     PAGES = {249--273},
      ISSN = {0024-3795},
   MRCLASS = {15A27 (16S50)},
  MRNUMBER = {1088666},
MRREVIEWER = {Robert M. Guralnick},
       URL = {https://doi.org/10.1016/0024-3795(91)90236-P},
}

@article {MTT,
    AUTHOR = {Motzkin, T. S. and Taussky, Olga},
     TITLE = {Pairs of matrices with property {$L$}. {II}},
   JOURNAL = {Trans. Amer. Math. Soc.},
  FJOURNAL = {Transactions of the American Mathematical Society},
    VOLUME = {80},
      YEAR = {1955},
     PAGES = {387--401},
      ISSN = {0002-9947},
   MRCLASS = {15.0X},
  MRNUMBER = {0086781},
MRREVIEWER = {H. Schwerdtfeger},
       URL = {https://doi.org/10.2307/1992996},
}

@article {HolbrookOmeara,
    AUTHOR = {Holbrook, John and O'Meara, K. C.},
     TITLE = {Some thoughts on {G}erstenhaber's theorem},
   JOURNAL = {Linear Algebra Appl.},
  FJOURNAL = {Linear Algebra and its Applications},
    VOLUME = {466},
      YEAR = {2015},
     PAGES = {267--295},
      ISSN = {0024-3795},
   MRCLASS = {15A27 (13A99 15A21)},
  MRNUMBER = {3278252},
       URL = {https://doi.org/10.1016/j.laa.2014.10.009},
}

@article {SethurSurvey,
    AUTHOR = {Sethuraman, B. A.},
     TITLE = {The algebra generated by three commuting matrices},
   JOURNAL = {Math. Newsl.},
  FJOURNAL = {Mathematics Newsletter},
    VOLUME = {21},
      YEAR = {2011},
    NUMBER = {2},
     PAGES = {62--67},
      ISSN = {0971-1694},
   MRCLASS = {16S50 (14M12 15A30)},
  MRNUMBER = {3013206},
}

@Misc{Bergman,
Author = {George M. Bergman},
Title = {Commuting matrices, and modules over Artinian local rings},
Year = {2013},
howpublished = {arXiv:1309.0053},
}

@article {HolOmla,
    AUTHOR = {Holbrook, John and Omladi\v{c}, Matja\v{z}},
     TITLE = {Approximating commuting operators},
   JOURNAL = {Linear Algebra Appl.},
  FJOURNAL = {Linear Algebra and its Applications},
    VOLUME = {327},
      YEAR = {2001},
    NUMBER = {1-3},
     PAGES = {131--149},
      ISSN = {0024-3795},
   MRCLASS = {15A27 (15A30)},
  MRNUMBER = {1823346},
MRREVIEWER = {Donald W. Robinson},
       URL = {https://doi.org/10.1016/S0024-3795(00)00286-X},
}

@article {Guralnick92,
    AUTHOR = {Guralnick, Robert M.},
     TITLE = {A note on commuting pairs of matrices},
   JOURNAL = {Linear and Multilinear Algebra},
  FJOURNAL = {Linear and Multilinear Algebra},
    VOLUME = {31},
      YEAR = {1992},
    NUMBER = {1-4},
     PAGES = {71--75},
      ISSN = {0308-1087,1563-5139},
   MRCLASS = {15A30},
  MRNUMBER = {1199042},
MRREVIEWER = {A.\ R.\ Wadsworth},
       DOI = {10.1080/03081089208818123},
       URL = {https://doi.org/10.1080/03081089208818123},
}

@article {NgoSivic14,
    AUTHOR = {Ngo, Nham V. and \v{S}ivic, Klemen},
     TITLE = {On varieties of commuting nilpotent matrices},
   JOURNAL = {Linear Algebra Appl.},
  FJOURNAL = {Linear Algebra and its Applications},
    VOLUME = {452},
      YEAR = {2014},
     PAGES = {237--262},
      ISSN = {0024-3795,1873-1856},
   MRCLASS = {14M99 (15A27)},
  MRNUMBER = {3201099},
MRREVIEWER = {Marcin\ Skrzy\'{n}ski},
       DOI = {10.1016/j.laa.2014.03.032},
       URL = {https://doi.org/10.1016/j.laa.2014.03.032},
}

@article {JelisiejewSivic22,
    AUTHOR = {Jelisiejew, Joachim and \v{S}ivic, Klemen},
     TITLE = {Components and singularities of {Q}uot schemes and varieties
              of commuting matrices},
   JOURNAL = {J. Reine Angew. Math.},
  FJOURNAL = {Journal f\"{u}r die Reine und Angewandte Mathematik. [Crelle's
              Journal]},
    VOLUME = {788},
      YEAR = {2022},
     PAGES = {129--187},
      ISSN = {0075-4102,1435-5345},
   MRCLASS = {14M99 (14J17 15A27 15A30)},
  MRNUMBER = {4445543},
MRREVIEWER = {Nham\ Vo\ Ngo},
       DOI = {10.1515/crelle-2022-0018},
       URL = {https://doi.org/10.1515/crelle-2022-0018},
}
\end{document}